\documentclass[a4paper,10pt]{elsarticle}
\usepackage{indentfirst}
\usepackage{graphicx}
\usepackage[colorlinks=true, bookmarksnumbered=true, bookmarksopen=true, bookmarksopenlevel=3, pdfstartview=FitH, linkcolor=blue, pdfmenubar=true, pdftoolbar=true, bookmarks=true,citecolor=blue, urlcolor=blue, filecolor=magenta,plainpages=false,pdfpagelabels,breaklinks]{hyperref}
\usepackage{mathrsfs}
\usepackage{amsbsy, amsmath, amsfonts,amssymb, amsthm,mathabx}
\usepackage{latexsym}

\usepackage{color}

\numberwithin{equation}{section}
\usepackage[lmargin=2.8cm,tmargin=3cm,bmargin=3cm,rmargin=2.8cm]{geometry}

\newtheorem{theorem}{Theorem}[section]

\newtheorem{proposition}[theorem]{Proposition}

\newtheorem{lemma}[theorem]{Lemma}
\newtheorem{remark}[theorem]{Remark}

\journal{submitted}

\begin{document}

	\begin{frontmatter}
		
		\title{Self-similar solutions for a superdiffusive heat equation with gradient nonlinearity}
		
		\author[a1]{Marcelo Fernandes de Almeida}
		\address[a1]{Department of Mathematics, Federal University of Sergipe,
			Avenue Rosa Else, S\~ao Crist\'ov\~ao, Sergipe, Brazil.}
		\ead{nucaltiado@gmail.com}
		%
		\author[a2]{Arl\'ucio Viana\corref{cor2}}
		\cortext[cor2]{Corresponding author}
		\address[a2]{{Department of Mathematics, Federal University of Sergipe, Avenue Vereador Ol\'impio Grande, Itabaiana, Sergipe, Brazil.}}
		\ead{arlucioviana@ufs.br}
		%

		\begin{abstract}
		This paper is devoted to global well-posedness, self-similarity and symmetries of solutions for a superdiffusive heat equation with superlinear and gradient nonlinear terms with initial data in new homogeneous Besov-Morrey type spaces. Unlike the heat equation, we need to develop an appropriate decomposition of the two-parametric Mittag-Leffler function in order to obtain Mikhlin-type estimates get our well-posedness theorem. To the best of our knowledge, the present work is the first one concerned with a well-posedness theory for a time-fractional partial differential equations of order $\alpha\in(1,2)$ with non null initial velocity.
		\end{abstract}
		
		
	\end{frontmatter}


	\section{Introduction}
	Let $\Delta_x$ be the Laplace operator $\sum_{i=1}^{N}\frac{\partial^2}{\partial x_i^{2}}$, $u:\mathbb{R}^{1+N}\rightarrow\mathbb{R}$ and $\partial^{\alpha}_t$ be the Caputo's fractional derivative of order $1<\alpha<2$ (see subsection \ref{Duhamel}). In this paper, we deal with the following equation
	\begin{equation}\label{heat-wave}
		\partial^{\alpha}_{t}u = \Delta_x u +\kappa_1\vert \nabla_x u\vert^{q} + \kappa_2\vert u\vert^{\rho-1}u,\; \kappa_1\neq 0,\, \kappa_2\in\mathbb{R} ,
	\end{equation}
	subject to the initial data
	\begin{equation}\label{initial-data}
		u(0,x)=\varphi(x) \text{ and } \partial_tu(0,x)=\psi(x), 
	\end{equation}
	where $q>1$ and $\rho>1$. Note that the rescaled function $u_{\gamma}(t,x):=\gamma^{\frac{2}{\rho-1}}u(\gamma^{\frac{2}{\alpha}}t,\gamma x)$ solves  (\ref{heat-wave}) with initial data
	 \begin{align}\label{key_au}
	 \varphi_{\gamma}(x)=\gamma^{\frac{2}{\rho-1}}\varphi(\gamma x) \text{ and } \psi_{\gamma}(x)=\gamma^{\frac{2}{\rho-1}+\frac{2}{\alpha}}\psi(\gamma x) ,
	 \end{align}
	 provided that $q=\frac{2\rho}{\rho+1}$ and $u(t,x)$ solves (\ref{heat-wave})-(\ref{initial-data}). Hence, we obtain a\textit{ scaling map} of solutions,
	 \begin{align}\label{scal-map}
	 u(t,x)\mapsto u_{\gamma}(t,x),\, \text{ for all }\,\gamma>0,
	 \end{align}
	 and solutions invariant by (\ref{scal-map}) will be called \textit{self-similar solutions}, that is, 
	 \begin{align}\label{self-similar}
	 u(t,x)=u_{\gamma}(t,x).
	 \end{align}
	In the study of self-similar solutions, the natural candidates to be initial data are the homogeneous functions,
	 \begin{align}\label{hom-data}
	 \varphi(\gamma x) =\gamma^{-\frac{2}{\rho-1}}\varphi(x) \text{ and } \psi(\gamma x) =\gamma^{-\frac{2}{\rho-1}-\frac{2}{\alpha}}\psi(x).\nonumber
	 \end{align}
	 In this work we are interested in existence of self-similar solutions to  (\ref{heat-wave})-(\ref{initial-data}). For this purpose, we study (\ref{heat-wave})-(\ref{initial-data}) through its integral formulation
	 \begin{equation}\label{int}
	 u(t,x) = G_{\alpha,1}(t)\varphi(x)+G_{\alpha,2}(t)\psi(x)+\mathcal{N}_{\alpha}(u)(t,x),
	 \end{equation}
	 where 
	 \begin{equation}\label{Mit-Lef}
	 \widehat{G_{\alpha,j}(t)f}(\xi) = t^{j-1}E_{\alpha,j}(-4\pi^2t^\alpha|\xi|^2)\widehat{f}(\xi),\; j=1,2,f\in\mathcal{S}'(\mathbb{R}^N) ,
	 \end{equation}
	 and
	 \begin{equation}
	 \mathcal{N}_{\alpha}(u)=\int_{0}^{t}G_{\alpha,1}(t-s)\int_{0}^{s}r_{\alpha}(s-\tau) \left(\kappa_2\vert u\vert^{\rho-1} u+\kappa_1\vert \nabla_x u\vert^{q}\right) d\tau ds.
	 \end{equation}
	 Here and hereafter a solution $u$ will be understood as a distribution $u(t,\cdot)$ satisfying (\ref{int}), for each $t>0$. 
	 
	 The presence of the gradient requires suitable estimates in certain Sobolev-Morrey spaces $\mathcal{M}^{s}_{r,\mu}$ and  this motivate us to study the problem in functional space
	 \begin{align}\label{norm}
	 \|u \|_{X_{\beta}} = \sup_{t>0} t^{\frac{\alpha}{2}+\beta}\|u(t)\|_{\mathcal{M}_{r,\mu}^1 } +\sup_{t>0} t^{\beta}\|u(t)\|_{\mathcal{M}_{r,\mu}},
	 \end{align}
where $\beta,r$ and $\mu$ will be chosen later (see (\ref{param1})). Assuming $\psi\notequiv0$ brings new difficulties, because we need to obtain suitable estimates for two-parametric Mittag-Leffler function $E_{\alpha,2}(4\pi^2t^{\alpha}\vert\xi\vert^2)$. More precisely, we develop an appropriate decomposition for Mittag-Leffler function in order to obtain a suitable estimate (see (\ref{point1}) and (\ref{point2})) which enables us to introduce the space
	\begin{align}
	\mathcal{I}=\{(\varphi,\psi)\in\mathcal{S}'\times\mathcal{S}'; (\varphi,\psi)\in D(\alpha,\beta)\times \tilde{D}(\alpha,\beta)\},
	\end{align}
where 
	\begin{align}
	D(\alpha,\beta):=\{\varphi\in \mathcal{S}';\, G_{\alpha,1}(t)\varphi\in X_{\beta}\} \text{ and } \widetilde{D}(\alpha,\beta):=\{\psi\in\mathcal{S};\, G_{\alpha,2}(t)\psi\in X_{\beta}\}, \nonumber
	\end{align}
	for all $t>0$. Hence, applying Lemma \ref{galpha} we obtain (see Remark \ref{rem1}-(B)) that $\mathcal{M}_{p,\mu}\times\mathcal{M}_{p,\mu}^{-2/\alpha}\subseteq D(\alpha,\beta)\times\widetilde{D}(\alpha,\beta)$. It's remarkable that the investigation of self-similarity and symmetries for (\ref{heat-wave})-(\ref{initial-data}), allows to deal with  following prototype functions
\begin{align}\label{prototype}
\varphi(x)=\epsilon_1\vert x\vert ^{-\frac{2}{\rho-1}} \text{ and } \psi(x)=\epsilon_2\vert x\vert^{-\frac{2}{\rho-1}-\frac{2}{\alpha}}.
\end{align} 
which belong to  $D(\alpha,\beta)\times\widetilde{D}(\alpha,\beta)$. Also, this functions can be  arbitrarily large in the space $L^2(\mathbb{R}^N)\times \dot{H}^{\frac{2}{\alpha}}(\mathbb{R}^N)$, see Remark \ref{rem2}-(A). 

Our symmetry result, roughly speaking, says that if the initial data $\varphi$ and $\psi$ are invariant on the orthogonal group acting on $\mathbb{R}^N$ so the solution is. In particular, we show the existence of radial self-similar solutions (see Remark \ref{rem3}-(A)).

We point out that our results holds true for $\alpha=1$ and $\psi=0$ and, in this case,  the upper bound  $(\gamma_2-\gamma_1)+\frac{N-\mu}{p_1}-\frac{N-\mu}{p_2}$ in Lemma \ref{galpha} can be removed. On the other hand, for  $1<\alpha<2$ the Mikhlin theorem yields more restrictive constraints to Lemma \ref{galpha} than the usual estimates for the heat semigroup in such a way that Theorem \ref{gw} cannot come near to $\alpha=2$. 
		
Now let us to review some works. Fujita \cite{YF} remarked that the  linear counterpart of equation (\ref{heat-wave}) has similarities with wave and heat equations and presents certain qualitative properties which qualifies it as a reasonable interpolation between these equations. When $\kappa_2=0$, $\alpha=1$ and $\psi=0$, (\ref{heat-wave})-(\ref{initial-data}) turns into the viscous Hamilton-Jacobi equation. Using scaling technique, Ben-Artzi {\it et al} \cite{bsweissler} found the number $r	_c=\frac{N(q-1)}{2-q}$ and showed that it is a critical exponent for existence of solutions in $L^r$. In particular, the problem is well-posed when $r\geq r_c$ and $1<q<2$. In the Remark \ref{rem1}-(C) we provide an existence result  for this problem in Morrey spaces $\mathcal{M}_{p,\,N-\frac{N}{r_c}p}$ which are strictly larger than the Lebesgue spaces, namely,
\begin{align} \label{inv-spaces1}
L^r(\mathbb{R}^N)\subsetneq\mathcal{M}_{p,\mu} (\mathbb{R}^N)\subsetneq D(1,\beta)
\end{align}
provided that $\frac{N}{r}=\frac{N-\mu}{p}$ and $p<r$.  Our existence result is then compatible with \cite[Theorem 2.1]{bsweissler}, in view of $1< p\leq r_c\leq r$. In particular, the initial data taken in Theorem \ref{gw} is larger than in \cite{bsweissler}.
	
Recently several authors have addressed the study of global existence, self-similarity, asymptotic self-similarity  and radial symmetry of solutions for the semilinear heat equation with gradient term, see e.g.  \cite{Chipot,gild,STW,Weissler-heat,Souplet1}. In \cite{STW} it is assumed that  $\varphi$ belongs to homogeneous Besov space $ \dot{B}_{r_1,\infty}^{-\beta_1}$ and
\begin{align}
\Vert \varphi\Vert_{\dot{B}_{r_1,\infty}^{-\beta_1}}=\sup_{t>0}t^{\beta_1/2}\Vert e^{t\Delta}\varphi\Vert_{L^{r_1}(\mathbb{R}^N)}\leq \epsilon,\;\; \beta_1=\frac{2}{\rho-1}-\frac{N}{r_1} .
\end{align}
By employing the Gagliardo-Nirenberg inequality, the authors studied the existence and asymptotic behavior of global mild solutions. On the other hand, our functional approach enables us to control the gradient estimates without making use of the Gagliardo-Nirenberg inequality and allows to deal with a larger class of functions space for initial data.

Let us now review some works concerning to (\ref{heat-wave})-(\ref{initial-data}) with $\psi=0$ and $\kappa_1=0$. In \cite{HMiao} the authors established $L^p-L^q$ estimates for $\{G_{\alpha,1}(t)\}_{t\geq 0}$ and showed a blowup alternative and local well-posedness in $L^q(\mathbb{R}^N)$-framework for any $\varphi\in L^q(\mathbb{R}^N)$, where $q\geq \frac{N\alpha (\rho+1)}{2}$. Making use of Mikhlin-Hormander's type theorem on Morrey spaces, the authors of \cite{Marcelo} studied self-similarity, symmetry, antisymmetry and positivity of global solutions with small data $\varphi\in\mathcal{M}_{p,\mu}$, $\mu=N-\frac{2p}{\rho-1}$.  In \cite{MJ}, the authors established existence of self-similarity, symmetries and asymptotic behavior of solutions in Besov-Morrey spaces $\mathcal{N}^{\sigma}_{p,\mu,\infty}$ and provided a maximal class of existence in the sense that there is no known results in $X\supsetneq \mathcal{N}^{\sigma}_{p,\mu,\infty}$. Indeed, the authors provided a larger class of initial data,  because 
\begin{align}\label{inv-spaces2}
\mathcal{M}_{p,\lambda}\subsetneq
\mathcal{N}_{p,\mu ,\infty }^{\sigma }\;\;\text{ and }\;\;\dot{B}_{r,\infty
}^{k}\subset\mathcal{N}_{p,\mu ,\infty }^{\sigma },
\end{align}%
where $\frac{N-\lambda }{p}=-\sigma +\frac{N-\mu }{p}=-k+\frac{N}{r}$, $\sigma =\frac{N-\mu }{p}-\frac{2}{\rho -1}$, $k=\frac{N
}{r}-\frac{2}{\rho -1}$ and $1\leq p<r$. All
spaces in (\ref{inv-spaces1}) and (\ref{inv-spaces2}) are invariant by scaling. 

We still observe that problem (\ref{heat-wave})--(\ref{initial-data}) can be studied with a Fourier multiplier $\sigma(D)$ in place of $\Delta_x$ , where $\vert\sigma(\xi)\vert\leq C\vert\xi\vert^k$ due to estimates (\ref{point1}) and (\ref{point2}) into Propositions \ref{fund-lemma} and \ref{fund-lemma2}. Example of such operator is Riesz potential $(-\Delta _{x})^{{k }/{2}}f=\mathcal{F}^{-1}|\xi |^{k}\mathcal{F}f$, where $\mathcal{F}$ denote the Fourier transform in $\mathcal{S}'$.


	The manuscript is organized as follows. Some basic properties of the Sobolev-Morrey
	spaces and Mittag-Leffler functions are reviewed in Section \ref{pre}. We state and make some remarks on our results in Section \ref{func-settings} and their proofs are performed in Section \ref{proofs}. Sections \ref{technical} and \ref{sme} are reserved to a careful study of the several estimates which are crucial to yield our results.
	
	\bigskip
	\section{Preliminaries}\label{pre}
	In this section we review some well-known properties
	of the Morrey spaces and  Sobolev-Morrey spaces, more details can be found in
	\cite{Sawano, Kato-Morrey,Yamazaki2,Mazzucato,Miyakawa1}. Also, we obtain an integral equation which is formally equivalent to  (\ref{heat-wave})-(\ref{initial-data}) in the lines of \cite{Kilbas2}.
	
	\subsection{Sobolev-Morrey spaces}\label{sms}
	Let $Q_{r}(x_{0})$ be the open ball in $\mathbb{R}^{N}$ centered at $x_{0}$
	and with radius $r>0$. Given two parameters $1\leq p<\infty $ and $0\leq \mu
	<N$, the Morrey space $\mathcal{M}_{p,\mu }=\mathcal{M}_{p,\mu }(\mathbb{R}%
	^{N})$ is defined to be the set of the functions $f\in L^{p}(Q_{r}(x_{0}))$ such
	that
	\begin{equation}
		\Vert f\Vert _{\mathcal{M}_{p,\mu}}:=\sup_{x_{0}\in \mathbb{R}^{n},\,r>0}r^{-\frac{%
				\mu }{p}}\Vert f\Vert _{L^{p}(Q_{r}(x_{0}))}<\infty,   \label{norm-Morrey}
	\end{equation}%
	which is a Banach space endowed with the norm (\ref{norm-Morrey}). For $s\in
	\mathbb{R}$ and $1\leq p<\infty ,$ the homogeneous Sobolev-Morrey space $%
	\mathcal{M}_{p,\mu }^{s}=(-\Delta_x)^{-s/2}\mathcal{M}_{p,\mu }$ is the Banach space of all tempered distributions $f\in\mathcal{S}'(\mathbb{R}^N)/\mathcal{P}$ modulo polynomials $\mathcal{P}$ with $N$ variables.
	If $s<\frac{N-\mu}{p}$ and $p>1$, from  \cite[Theorem 1.1]{Sawano} or \cite{Mazzucato}, it holds that
	\begin{equation}
		\Vert f\Vert_{\mathcal{M}_{p,\mu}}\sim \left\Vert\left(\sum_{\nu\in\mathbb{Z}}\vert \mathcal{F}^{-1}\psi_{\nu}(\xi)\mathcal{F}f\vert^2\right)^{\frac{1}{2}}\right\Vert_{\mathcal{M}_{p,\mu}}\label{Littewood-Paley-Morrey},
	\end{equation}
	where $\sim$ denotes norm equivalence and $\{\psi_{\nu}\}_{\nu\in\mathbb{Z}}$ is a homogeneous Littlewood-Paley resolution of unity, that is, 
	\begin{align}
		\psi_{\nu}(\xi)=\phi_{\nu}(\xi)-\phi_{\nu-1}(\xi),\;\; \phi_{\nu}(\xi)=\phi_0(\xi/2^{\nu})\nonumber,
	\end{align}
	for $\phi_0\in C^{\infty}_{0}(\mathbb{R}^N)$ such that $\phi_0=1$ on the ball $Q_1(0)$ and $\text{supp}\,\phi_0\,\subset Q_{2}(0)$. In particular, using (\ref{Littewood-Paley-Morrey})  and that $\vert \xi\vert ^s\sim2^{s\nu}$ on the $\text{supp}\,\psi_{\nu}(\xi)\subset \{ \xi\in\mathbb{R}^N\,:\, 2^{\nu-1}< \vert \xi\vert< 2^{\nu+1}\}$,  we obtain
	\begin{align}
		\left\Vert\left(\sum_{\nu\in\mathbb{Z}}\vert 2^{s\nu} \mathcal{F}^{-1}\psi_{\nu}(\xi)\mathcal{F}f\vert^2\right)^{\frac{1}{2}}\right\Vert_{\mathcal{M}_{p,\mu}}&\sim \left\Vert\left(\sum_{\nu\in\mathbb{Z}}\vert  \mathcal{F}^{-1}\psi_{\nu}(\xi)\vert \xi\vert^s\mathcal{F}f\vert^2\right)^{\frac{1}{2}}\right\Vert_{\mathcal{M}_{p,\mu}}\nonumber\\
		&= \left\Vert\left(\sum_{\nu\in\mathbb{Z}}\vert  2^{\nu N}\widecheck{\psi}(2^{\nu}\cdot)\ast (\vert\xi\vert^s\widehat{f})^{\vee}\vert^2\right)^{\frac{1}{2}}\right\Vert_{\mathcal{M}_{p,\mu}}\nonumber\\
		&\sim \Vert (\vert\cdot\vert^s\widehat{f})^{\vee}\Vert_{\mathcal{M}_{p,\mu}}. \label{norm-key-SM}
	\end{align}
	Given $f\in\mathcal{M}_{p,\mu}^s$, the quantity (\ref{norm-key-SM}) define two equivalent norms on Sobolev-Morrey space, namely, 
	\begin{equation}
		\left\Vert f\right\Vert _{\mathcal{M}_{p,\mu }^{s}}=\left\Vert (\vert\cdot\vert^s\widehat{f})^{\vee}\right\Vert _{\mathcal{M}_{p,\mu}} \text{or }  \left\Vert f\right\Vert _{\mathcal{M}_{p,\mu }^{s}}=\left\Vert\left(\sum_{\nu\in\mathbb{Z}}\vert 2^{s\nu} \mathcal{F}^{-1}\psi_{\nu}(\xi)\mathcal{F}f\vert^2\right)^{\frac{1}{2}}\right\Vert_{\mathcal{M}_{p,\mu}}\label{norm-SM}.
	\end{equation}
	It follows from  Littlewood-Paley decomposition of the Lebesgue space $L^p(\mathbb{R}^N)$ and homogeneous Sobolev space $H^{s}_p(\mathbb{R}^N)$ that $\mathcal{M}_{p,0}=L^{p}(\mathbb{R}^N)$ and $\mathcal{M}_{p,0}^{s}=\dot{H}_{p}^{s}(\mathbb{R}^N)$, respectively.  Also, Morrey and Sobolev-Morrey spaces present the following scaling
	\begin{equation}
		\Vert f(\gamma \cdot )\Vert _{\mathcal{M}_{p,\mu}}=\gamma^{-\frac{N-\mu }{p}}\Vert
		f\Vert _{\mathcal{M}_{p,\mu}} \; \text{ and }\;
		\left\Vert f(\gamma\cdot )\right\Vert _{\mathcal{M}_{p,\mu }^{s}}=\gamma
		^{s-\frac{N-\mu }{p}}\left\Vert f\right\Vert _{\mathcal{M}_{p,\mu }^{s}}%
		\text{,}  \label{scal-SM}
	\end{equation}%
	where the exponents $s$ and $s-\frac{N-\mu }{p}$ are called \textit{scaling index} and \textit{regularity index}, respectively.
	\begin{lemma}
		\label{lem:2.1} Suppose that $s\in\mathbb{R}$, $1\leq
		p_1,p_2,p_3<\infty$ and $0\leq\mu_i<N$, $i=1,2,3$.
		
		\begin{description}
			\item[(i)] (Inclusion) If $\frac{N-\mu_1}{p_1}=\frac{N-\mu_2}{p_2}$ and $%
			p_1\leq p_2$,
			\begin{align}
				\mathcal{M}_{p_2,\mu_2}\subset\mathcal{M}_{p_1,\mu_1}.  \label{emb1}
			\end{align}
			
			\item[(ii)] (Sobolev-type embedding) Let $p_1\leq
			p_2$, 
			\begin{equation}
				\mathcal{M}_{p_1,\mu}^{s}\subset\mathcal{M}_{p_2,\mu}^{s - \left(\frac{N-\mu}{p_1}-\frac{N-\mu}{p_2}\right)}.  \label{sobolev-emb}
			\end{equation}
			
			\item[\textbf{(iii)}] (H\"oder inequality) Let $\;\frac{1}{p_3}=\frac{1}{p_2}+%
			\frac{1}{p_1}$ and $\frac{\mu_3}{p_3}=\frac{\mu_2}{p_2}+\frac{\mu_1}{p_1}$.
			If $f_j\in \mathcal{M}_{p_j,\mu_j}$ with $j=1,2$, then $f_1f_2\in\mathcal{M}%
			_{p_3,\mu_3}$ and
			\begin{equation}
				\Vert f_1f_2\Vert_{p_3,\mu_3}\leq \Vert f_1\Vert_{p_1,\mu_1}\Vert
				f_2\Vert_{p_2,\mu_2}.  \label{eq:holder}
			\end{equation}
		\end{description}
	\end{lemma}
	
	Finally, notice that the following homogeneous functions of degree $-d$ and $s-d$, belong to Morrey and Sobolev-Morrey spaces, respectively:
	\begin{align}\label{data-stand}
		\rho_0(x)=Y_{k}(x)|x|^{-d-k}\in \mathcal{M}_{p,\mu} \text{ and } \rho_s(x)=Y_{k}(x)|x|^{s-d-k}\in \mathcal{M}_{p,\mu}^s,
	\end{align}
	where $Y_{k}(x)\in L^{p}(\mathbb{S}^{N-1})$ is a harmonic homogeneous polynomial of degree $k,$ $\mu=N-dp$, $0<d-s<N$ and $1<p<N/d$. Indeed, using Theorem 4.1 in \cite[Ch. 4]{Stein1} we obtain $\widehat{\rho}_s(\xi)=\gamma_{k,s}Y_{k}(\xi)\vert \xi\vert^{d-s-k-N}$ provided $0<d-s<N$, where $\gamma_{k,s}$ is a positive constant.  It follows  from (\ref{norm-SM}) that
	\begin{align}
		\Vert \rho_s\Vert_{\mathcal{M}_{p,\mu}^s}&=\left\Vert \left(\sum_{\nu=-\infty}^{+\infty}\vert2^{s\nu} \mathcal{F}^{-1}\psi_{\nu}(\xi)\gamma_{k,s}Y_{k}(\xi)\vert\xi\vert^{d-s-k-N}\vert^2\right)^{\frac{1}{2}}\right\Vert_{\mathcal{M}_{p,\mu}}\nonumber\\
		&\sim\left\Vert \left(\sum_{\nu=-\infty}^{+\infty}\vert \mathcal{F}^{-1}\psi_{\nu}(\xi)\vert \xi\vert^{s}\gamma_{k,s}Y_{k}(\xi)\vert\xi\vert^{d-s-k-N}\vert^2\right)^{\frac{1}{2}}\right\Vert_{\mathcal{M}_{p,\mu}}\nonumber\\
		&=\Vert \rho_0\Vert_{\mathcal{M}_{p,\mu}},
	\end{align}
	which is finite. In fact,  polar coordinates in $\mathbb{R}^N$ and homogeneity of $Y_{k}(x)\in L^{p}(\mathbb{S}^{N-1})$ yield
	\begin{align*}
		\Vert \rho_0\Vert_{L^{p}(Q_r)}^p=\int_{\mathbb{S}^{N-1}}\vert Y_k(x')\vert^p\int_0^r t^{N-dp-1}dt\,d\sigma(x')=\Vert Y_k\Vert_{L^p(\mathbb{S}^{N-1})}^p\;r^{\mu} ,
	\end{align*}
	where $\mu=N-dp$, $1<p<N/d$. 
	
	\subsection{Duhamel formula}\label{Duhamel}

	We consider the partial fractional differential equation
	\begin{align}\label{eq_gen1}
		\begin{cases}
			\partial_t^{\alpha}u(t,x)-\Delta_x u(t,x)=f(t,x),\; x\in\mathbb{R}^N,\, t>0,\\
			u(t,x)\bigr\vert_{t=0}= \varphi(x) \text{ and }\frac{\partial }{\partial t}u(t,x)\Bigr\vert_{t=0}= \psi(x),
		\end{cases}
	\end{align}
	for $\alpha\in (1,2)$ and $\partial_t^{\alpha}$ stands for partial fractional derivative given by
	\begin{align}
		\partial^{\alpha}_{t}f(t,x)=\frac{1}{\Gamma(m-\alpha)}\int_0^t\frac{\partial_s^mf(s,x)}{(t-s)^{\alpha+1-m}}ds, \;\; m-1<\alpha \leq m,\; m\in\mathbb{N}.\nonumber
	\end{align}
	Formally, applying the Fourier transform in (\ref{eq_gen1}), we obtain the fractional ordinary differential equation 
	\begin{align*}
		\begin{cases}
			\partial_t^{\alpha}\widehat{u}(t,\xi)-4\pi^2\vert\xi\vert^2\widehat{u}(t,\xi)=\widehat{f}(t,\xi),\\
			\widehat{u}(t,\xi)\vert_{t=0}= \widehat{\varphi}(\xi)\text{ and } \partial_t\widehat{u}(t,\xi)\bigr\vert_{t=0}= \widehat{\psi}(\xi)
		\end{cases}
	\end{align*}
	which is equivalent to 
	\begin{align}
		\widehat{u}(t,\xi)= &E_{\alpha,1}(-4\pi^2t^{\alpha}\vert\xi\vert^2)\widehat{\varphi} + tE_{\alpha,2}(-4\pi^2t^{\alpha}\vert\xi\vert^2)\widehat{\psi}+\nonumber\\
		&+\int_0^t E_{\alpha,1}(-4\pi^2(t-s)^{\alpha}\vert\xi\vert^2)\int_0^{s}r_{\alpha}(s-\tau)\widehat{f}(\tau,\xi)d\tau ds\nonumber,
	\end{align} 
	where $E_{\alpha,\beta}(z)$ denotes the two-parametric Mittag-Leffler function
	\begin{equation}
		E_{\alpha,\beta}(z)=\sum_{k=0}^{\infty}\frac{z^k}{\Gamma(\alpha k+\beta)}\text{ and }\; E_{\alpha}(z):=E_{\alpha,1}(z),\; \text{for all } \alpha, \beta>0.
		\label{mit1}
	\end{equation}
	Hence, in original variables, we have
	 \begin{equation}
	 u(t,x) = G_{\alpha,1}(t)\varphi(x)+G_{\alpha,2}(t)\psi(x)+\mathcal{N}_{\alpha}(u)(t,x),
	 \end{equation}
where $	\widehat{G_{\alpha,j}(t)f}(\xi)$ is given by (\ref{Mit-Lef}) and $f(u)=\kappa_2\vert u\vert^{\rho-1} u+\kappa_1\vert \nabla_x u\vert^{q}$.

Note that  $G_{2,2}(t)$ is the  wave group $\left(\frac{\sin(4\pi^2t\vert \xi\vert)}{4\pi^2t\vert \xi\vert}\right)^{\vee},$ $\,G_{2,1}(t)=\left(\cos (4\pi^2t\vert\xi\vert)\right)^{\vee}$ and $G_{1,1}(t)=(e^{-4\pi^2t\vert\xi\vert^2})^{\vee}$.

\section{Functional setting and theorems}\label{func-settings}
	
	Before starting our theorems, let $\beta>0$ and $0\leq\mu<N$ stand for
	\begin{equation}
		\beta=\frac{\alpha}{2}\left(\frac{N-\mu}{p}-\frac{N-\mu}{r}\right) \text{ and }\, \mu = N -\frac{2p}{\rho-1}\label{param1} ,
	\end{equation}
	which make $\Vert\cdot\Vert_{X_\beta}$ be invariant by \textit{scaling map} (\ref{scal-map}).

	\subsection{Well-posedness}
	Given a Banach space $Y$, we will denote $B_Y(\varepsilon)$ a closed ball of radius $\varepsilon$ centered at the origin of the space $Y$. 
	
	Our well-posedness results is stated as follows.
	\begin{theorem}[Well-posedness]\label{gw} Let $N\geq2$, $1<\alpha<2$, $q=\frac{2\rho}{\rho+1}$, and $0\leq \mu=N-\frac{2p}{\rho-1}$, for $p>1$. Suppose that $\frac{N-\mu}{p}-\frac{N-\mu}{r}<2$, $r>\rho>1+\alpha$,
		\begin{equation}
			\frac{p}{r}<\frac{1}{\alpha}-\frac{1}{2},\;\;\;\;\;\;\frac{\alpha}{2-\alpha}<q<\frac{2}{\alpha},\;\;\;\;\;\;\left(1-\frac{p}{r}\right) < \frac{\rho-1}{\alpha}\left(\frac{1}{q}-\frac{\alpha}{2}\right)\label{param-hip2} .
		\end{equation}
		
		\begin{itemize}
			\item[(i)]\textbf{(Global existence)} There exist $\varepsilon>0$ such that if $\Vert \varphi\Vert_{D(\alpha,\beta)} +\Vert \psi\Vert_{\widetilde{D}(\alpha,\beta)}\leq \varepsilon$, the problem  (\ref{heat-wave})-(\ref{initial-data}) has a unique global-in-time mild solution $u\in B_{X_{\beta}}(2\varepsilon)\subset X_{\beta}$ satisfying 
			\begin{equation}
				\Vert u(t,\cdot)\Vert_{\mathcal{M}_{r,\mu}}\leq C t^{-\beta}\; \text{ and }\; \Vert \nabla_x u(t,\cdot)\Vert_{\mathcal{M}_{r,\mu}}\leq C t^{-\beta-\alpha/2} .
			\end{equation}
			\item[(ii)](\textbf{Stability in $X_{\beta}$}) The solution $u$ of Theorem \ref{gw}(i) is stable related to the initial data $\varphi$ and $\psi$, that is, the data-map solution $(\varphi,\psi)\mapsto u$ is locally Lipschitz continuous form $D(\alpha,\beta)\times\widetilde{D}(\alpha,\beta)$ to $X_{\beta}$,
			\begin{align}
				\Vert u-\tilde{u}\Vert_{X_{\beta}}\leq C\left(\Vert \varphi-\tilde{\varphi} \Vert_{D(\alpha,\beta)}+\Vert \psi-\tilde{\psi} \Vert_{\widetilde{D}(\alpha,\beta)}\right) ,
			\end{align}
			where $u$ and $\tilde{u}$ are solutions of (\ref{heat-wave}) with initial values $( \varphi,\psi)$ and $(\tilde{\varphi},\tilde{\psi})$, respectively.
		\end{itemize}
	\end{theorem}
	\begin{remark}\label{rem1}Let us compare our theorem with some previous results.
		\begin{itemize}
			\item[(A)] If $\psi=0$, we may take $\varphi\in \mathcal{M}_{p,\mu}$ in Theorem \ref{gw}-(i) with smallness on $\Vert\varphi\Vert_{\mathcal{M}_{p,\mu}}$.
			\item[(B)] In Theorem \ref{gw}-(i) holds for $\alpha=1$ and $\psi=0$. Hence, the space $D(1,\beta)$ strictly includes the space $\mathcal{N}(\varphi)$ taken in \cite{STW}. Indeed, let $r<r_1$ and $\mu_2=0$ in Lemma \ref{lem:2.1}-(i) to get
			\begin{align}
				\Vert \varphi\Vert_{D(1,\beta)}&=\sup_{t>0} t^{\beta}\Vert e^{t\Delta}\varphi\Vert_{\mathcal{M}_{r,\mu}} + \sup_{t>0} t^{\frac{1}{2}+\beta}\Vert e^{t\Delta}\varphi\Vert_{\mathcal{M}_{r,\mu}^1},\nonumber\\
				&\leq \sup_{t>0} t^{\beta}\Vert e^{t\Delta}\varphi\Vert_{L^{r_1}} + \sup_{t>0} t^{\frac{1}{2}+\beta}\Vert e^{t\Delta}\varphi\Vert_{\dot{H}_{r_1}^1}=\Vert \varphi\Vert_{\mathcal{N}(\varphi)}.\nonumber
			\end{align}
			On the one hand (see \cite[(2.56)]{Mazzucato2}), homogeneous Besov-Morrey spaces can be defined by
			\begin{align}
				\mathcal{N}_{r,\mu,\infty}^{-2s}=\left\{ f\in\mathcal{S}';\, \Vert f\Vert_{\mathcal{N}_{r,\mu,\infty}^{-2s}}=\sup_{t>0}t^{-s}\Vert e^{t\Delta}f\Vert_{\mathcal{M}_{r,\mu}}<\infty\right\},\; s>0.\nonumber
			\end{align}
			Hence, the space $D(1,\beta)$ is a kind of Besov-Morrey spaces. On the other hand, when $\alpha\neq 1$ the norms $\Vert \varphi\Vert_{D(\alpha,\beta)}=\Vert G_{\alpha,1}(t)\varphi\Vert_{X_{\beta}}$ and $\Vert \psi\Vert_{\widetilde{D}(\alpha,\beta) }=\Vert G_{\alpha,2}(t)\psi\Vert_{X_{\beta}}$ satisfy
			\begin{align}
				\Vert \varphi\Vert_{D(\alpha,\beta)}\leq C  \Vert \varphi\Vert_{\mathcal{M}_{p,\mu}}\; \text{ and }\;
				\Vert \psi\Vert_{\widetilde{D}(\alpha,\beta)}\leq  C \Vert \psi\Vert_{\mathcal{M}_{p,\mu}^{-{2}/{\alpha}}}\nonumber
			\end{align}
			in view of Lemma \ref{galpha}, so $\mathcal{M}_{p,\mu}\subset D(\alpha,\beta)$ and $\mathcal{M}_{p,\mu}^{-{2}/{\alpha}}\subset \widetilde{D}(\alpha,\beta)$.
			
			\item[(C)] (Viscous Hamilton-Jacobi) Let $\kappa_2=0$, $\psi=0$ in (\ref{heat-wave})-(\ref{initial-data}), $\mu=N-\frac{q-1}{2-q}p$ and $\Vert\varphi\Vert_{D(\alpha,\beta)}$ small enough. Using the proof of Theorem \ref{gw}, the problem (\ref{heat-wave})-(\ref{initial-data}) has a unique solution $u\in C((0,\infty);\mathcal{M}_{r,\mu})\cap C((0,\infty);\mathcal{M}_{r,\mu}^1)$ such that
			\begin{equation}
				\sup_{t>0}t^{\frac{(N-\mu)}{2}\alpha\left(\frac{1}{p}-\frac{1}{r}\right)}\|u(t)\|_{\mathcal{M}_{r,\mu}} \leq C\; \mbox{ and }\; \sup_{t>0}t^{\frac{\alpha}{2}+\frac{(N-\mu)}{2}\alpha\left(\frac{1}{p}-\frac{1}{r}\right)}\|\nabla u(t)\|_{\mathcal{M}_{r,\mu}} \leq C \nonumber
			\end{equation}
			under the assumptions in Theorem \ref{gw} with the change $\rho=\frac{2}{2-q}$.
			In other words, we obtain a version of Theorem 2.1 and Proposition 2.3 of \cite{bsweissler}	when  $1<\alpha<2$. If $\alpha=1$, the assumption $\frac{N-\mu}{p}-\frac{N-\mu}{r}<2$ is not necessary  due to the smoothing effect of the heat semigroup in $\mathcal{M}_{p,\mu}$ (see e.g. \cite{Kato-Morrey}).
		\end{itemize}
	\end{remark}
	
\subsection{Self-similar solutions}
	As we commented before, a necessary condition for initial data to produce self-similar solutions is homogeneity and simplest candidates are the radial functions (\ref{prototype}). Hence, we need that $D(\alpha,\beta)$ and $\widetilde{D}(\alpha,\beta)$ satisfies
	\begin{align}
		\Vert \psi_{\gamma}\Vert_{\widetilde{D}(\alpha,\beta)} = \Vert \psi\Vert_{\widetilde{D}(\alpha,\beta)}\; \text{ and }\; \Vert \varphi_{\gamma}\Vert_{D(\alpha,\beta)} = \Vert \varphi\Vert_{D(\alpha,\beta)} ,
	\end{align}
which comes from the scaling invariance of $X_\beta$.
	
	\smallskip
\begin{theorem}[Self-similarity]\label{selfsimilarity} Under the assumptions of Theorem \ref{gw}-(i), let $\varphi$ and $\psi$ be homogeneous functions of degree $-\frac{2}{\rho-1}$ and $-\frac{2}{\rho-1}-\frac{2}{\alpha}$, respectively. Then the solution $u$ of Theorem \ref{gw}-(i) is self-similar.
\end{theorem}

Notice that (\ref{data-stand}) and Remark \ref{rem1}-(B) permit us to take the {\it singular functions}
\begin{align}
	\varphi(x)= \varepsilon_1 \vert x\vert ^{-\frac{2}{\rho-1}} \text{ and } \psi(x)= \varepsilon_2 \vert x\vert ^{-\frac{2}{\rho-1}-\frac{2}{\alpha}},\label{key_sing}
\end{align} 
as initial data, since $\varphi\in \mathcal{M}_{p,\mu}\subset D(\alpha,\beta)$ and $\psi\in \mathcal{M}^{-2/\alpha}_{p,\mu}\subset \widetilde{D}(\alpha,\beta)$ provided $\mu = N -\frac{2p}{\rho -1}$, $\rho>\max \{1+\frac{2}{N}, 1+\frac{2\alpha}{\alpha N-2}\}$ and $1<p<r$.
	
\begin{remark}\label{rem2}Let us remark some consequences of this theorem.
	\begin{itemize}
			\item[(A)](Infinity energy data) In Theorem \ref{selfsimilarity} we can build singular initial data $(\psi,\varphi)$ which can be arbitrarily large in  $L^2(\mathbb{R}^N)\times \dot{H}^{2/\alpha}(\mathbb{R}^N)$, provided that  $\frac{2}{\alpha}+\frac{2}{\rho-1}<\frac{N}{2}$ and $1<p<\frac{N(\rho-1)}{2}$. 
			Indeed, let  $\varphi\in\mathcal{S}'(\mathbb{R}^N)$ and $\psi\in\mathcal{S}'(\mathbb{R}^N)/\mathcal{P}$ given by (\ref{key_sing}). Using $\widehat{\varphi}(\xi)=\gamma_{0,0}\,\varepsilon_1 \vert \xi\vert ^{\frac{2}{\rho-1}-N}$, we see that $\varphi$ and $\psi$ are arbitrarily large in $\dot{H}^{2/\alpha}$ and $L^2$ in view of 	
		\begin{align}
				\Vert \psi\Vert_{L^2(\mathbb{R}^N)}^2&=\varepsilon_2^2 \int_{\mathbb{R}^N}\vert x\vert ^{-\frac{4}{\rho-1}-\frac{4}{\alpha}}dx\nonumber\\
				&=\varepsilon_2^2\lim_{\omega_2\rightarrow\infty} \int_{0}^{\omega_2}\int_{\mathbb{S}^{N-1}}r ^{-\frac{4}{\rho-1}-\frac{4}{\alpha}}r^{N-1}d\sigma dr=C\lim_{\omega_2\rightarrow\infty}\omega_2 ^{-\frac{4}{\rho-1}-\frac{4}{\alpha}+N}=+\infty\nonumber
		\end{align}
			and 
		\begin{align}
				\Vert \varphi\Vert_{\dot{H}^{\frac{2}{\alpha}}(\mathbb{R}^N)}^2&=\int_{\mathbb{R}^N}\vert\xi\vert^{4/\alpha}\vert\widehat{\varphi}(\xi)\vert^2d\xi=\gamma_{0,0}^2\varepsilon_1^2\int_{\mathbb{R}^N}\vert\xi\vert^{4/\alpha+\frac{4}{\rho-1}-2N}d\xi\nonumber\\
				&=C\lim_{\omega_1\rightarrow 0}\int_{\omega_1}^{\infty}\int_{\mathbb{S}^{N-1}}r ^{4/\alpha+\frac{4}{\rho-1}-N-1}d\sigma dr=C\lim_{\omega_1\rightarrow 0}\omega_1 ^{\frac{4}{\alpha}+\frac{4}{\rho-1}-N}=+\infty.\nonumber
		\end{align}
			Then, even the initial data $\varphi$ and $\psi$ live in Morrey spaces $\mathcal{M}_{p,\mu}$ and $\mathcal{M}_{p,\mu}^{-{2}/{\alpha}}$, respectively, they may be arbitrarily large in $\dot{H}^{2/\alpha}(\mathbb{R}^N)$ and $L^2(\mathbb{R}^N)$.
			
			\item[(B)] Inspired by \cite{Ribaud}, we use a  Littlewood-Paley decomposition of the Sobolev-Morrey spaces (see subsection \ref{sms}) to build general singular functions for Theorem \ref{selfsimilarity}. In fact,
		let $Y_{k_1}(x)$, $Y_{k_2}(x)$  be homogeneous harmonic polynomials  of degree $k_1$ and  $k_2$, respectively. Consider $S(\varphi,\psi)$ the set of functions $(\varphi,\psi)\in\mathcal{S}'(\mathbb{R}^N)\times \mathcal{S}'(\mathbb{R}^N)/\mathcal{P}$ such that 
			\begin{equation}\label{harmon-initial}
				\varphi(x)=\epsilon_1\frac{Y_{k_1}(x)}{\vert x\vert^{\frac{2}{\rho-1}+k_1} }\text{ and }  \psi(x)=\epsilon_2\frac{Y_{k_2}(x)}{\vert x\vert^{\frac{2}{\rho-1}+\frac{2}{\alpha}+k_2}}.
			\end{equation}
		By (\ref{data-stand}), the set $S(\varphi,\psi)$ gives us a class of data  for existence of self-similar solutions for (\ref{heat-wave})-(\ref{initial-data}). 
	\end{itemize}
\end{remark}
	
	\subsection{Symmetries}
	This subsection concerns with symmetries of solutions obtained in Theorems \ref{gw} and \ref{selfsimilarity}. It is straightforward to check that $E_{\alpha,1}(4\pi^2t^{\alpha}\vert\xi\vert^2)$ and $tE_{\alpha,2}(4\pi^2t^{\alpha}\vert\xi\vert^2)$ is invariant by the set $\mathcal{O}(N)$ of all rotations in $\mathbb{R}^{N}$. It follows that $G_{\alpha,1}(t)$ and $G_{\alpha,2}(t)$ are  $\mathcal{O}(N)-$ invariant.  Hence, it is natural to ask whether or not the solutions of the above theorems, present symmetry properties under certain qualitative conditions of the initial data.
	
	Let $\mathcal{A}$ be a subset of $\mathcal{O}(N).$ A function $h$ is said symmetric under action $\mathcal{A}$ when $h(x)=h(T(x))$ for all $T\in\mathcal{A}$. If $h(x)=-h(T(x))$, the function $h$ is said antisymmetric under the action of $\mathcal{A}$.
	
	\smallskip
	\begin{theorem}\label{symmetry} Let the hypotheses of Theorem \ref{gw} be satisfied . The solution $u(\cdot,t)$ is symmetric  for all $t>0$, whenever $\varphi$ and $\psi$ are symmetric  under action $\mathcal{A}$.
	\end{theorem}

	\begin{remark}\label{rem3}A radially symmetric solution is a self-similar solution, if the profile $\omega$ depends only on $r=\vert x\vert$, that is, there is a function $\mathcal{U}$ such that 
		$u(t,x)= t^{-\frac{\alpha}{\rho-1}} \mathcal{U}\left({\vert x\vert}/{t^{\frac{\alpha}{2}}}\right), \;\; t>0.$
		\begin{itemize}
			\item[(A)] Let $\mathcal{A}=\mathcal{O}(N)$ in Theorem \ref{symmetry}. If $\varphi$ and $\psi$ are radial and homogeneous functions of degree $-\frac{2}{\rho-1}$ and $-\frac{2}{\rho-1}-\frac{2}{\alpha}$, respectively (see Remark \ref{rem2}), then Theorems \ref{gw}, \ref{selfsimilarity} and \ref{symmetry} imply that (\ref{heat-wave})-(\ref{initial-data}) have a unique self-similar solution $u\in X_\beta$ which is radially symmetric in $\mathbb{R}^N$.
			
			\item[(B)] Unlike the case $\kappa_1=0$, antisymmetry does not hold in general, for $\kappa_1\neq0$.\end{itemize}
	\end{remark}
	
	\section{Technical estimates}\label{technical}
	In this section we prove some estimates of Mikhlin type for Mittag-Leffler functions. In spite of the fact that these estimates are necessary in the proof of Lemma \ref{galpha}, they are of independent interest. We start the section with a suitable  decomposition of $E_{\alpha,\beta}(z)$.
	
	\subsection{Decompositions of $E_{\alpha,\beta}(z)$}
	
	\begin{proposition}\label{dec-gen-E}
		Let $z\in\mathbb{C}$ be such that $\mathcal{R}e (z) >0$ and define
		\begin{equation}\label{axi0}
			\omega_{\alpha,\beta}(z)=\frac{z^{\frac{1-\beta}{\alpha}}}{\alpha}\left[\exp\left(a_\alpha(z)+\frac{1-\beta}{\alpha}\pi i\right) + \exp\left(b_\alpha(z)-\frac{1-\beta}{\alpha}\pi i\right)\right]
		\end{equation}
		and 
		\begin{equation}\label{axi}
			l_{\alpha,\beta}(z)=\int_0^\infty \mathcal{H}_{\alpha,\beta}(s)e^{- z^{\frac{1}{\alpha}}s^{\frac{1}{\alpha}}}z^{\frac{1}{\alpha}(1-\beta)}ds,
		\end{equation}
		where 
		\begin{equation}
			\mathcal{H}_{\alpha,\beta}(s)=\frac{1}{\alpha\pi}\;\frac{\sin[(\alpha-\beta)\pi]-s\sin(\beta\pi)}{s^2+2s\cos(\alpha \pi)+1} s^{\frac{1-\beta}{\alpha}} , \label{axi2}
		\end{equation}
		\begin{equation}
		 a_\alpha(z)=z^{\frac{1}{\alpha}}e^{\frac{\pi i}{\alpha}}\ \ \text{and} \ \ b_\alpha(z)= z^{\frac{1}{\alpha}}e^{-\frac{\pi i}{\alpha}} .\nonumber
		\end{equation}
		Suppose that $1<\alpha<2$ and $1\leq\beta\leq2$, then 
		\begin{equation}\label{decomp}
			E_{\alpha,\beta}(-z)=\omega_{\alpha,\beta}(z)+ l_{\alpha,\beta}(z) .
		\end{equation}
	\end{proposition}
	
	\begin{proof}
		Recall that Mittag-Leffler function can be defined by
		\begin{equation}
			E_{\alpha,\beta}(-z) = \frac{1}{2\pi i}\int_{Ha}\frac{t^{\alpha-\beta}e^t}{t^\alpha+z} dt ,\label{M-L}
		\end{equation}
		where  $Ha$ is the Hankel path, i.e. a path starts and ends at $-\infty$ and encircles the circular disk $\vert t\vert\leq \vert z\vert^{\frac{1}{\alpha}}$ positively. The integrand $\Phi(t)=\frac{t^{\alpha-\beta}e^t}{t^\alpha+z}$ of (\ref{M-L}) has two poles $a_\alpha(z)$ and $b_\alpha(z)$, because $1<\alpha<2$. Proceeding as in  \cite[Lemma 1.1]{YF}, the residues theorem yields 
		\begin{align*}
			2\pi i E_{\alpha,\beta}(-z) &= \int_{\infty}^{R}\Phi(te^{-\pi i})d(te^{-\pi i}) + 2\pi i \left(Res(\Phi,a_\alpha(z)) + Res(\Phi,b_\alpha(z))\right)\\
			& -\int_{\epsilon}^{R}\Phi(te^{-\pi i})d(te^{-\pi i})- \int_{R}^{\epsilon}\Phi(te^{\pi i})d(te^{\pi i})\\ 
			&- \int_{-\pi}^{\pi}\Phi(\epsilon e^{\theta i})d(\epsilon e^{\theta i})+ \int_{R}^{\infty}\Phi(te^{\pi i})d(te^{\pi i}) \\
			& := I_1(R) +  2\pi i \left(Res(\Phi,a_\alpha(z)) + Res(\Phi,b_\alpha(z))\right) -I_2(\epsilon,R) - I_3(R,\epsilon) \\
			& - I_4(\epsilon) + I_5(R).
		\end{align*}
		We first get
		\begin{equation}
			\lim_{R\rightarrow\infty} I_1(R) = \lim_{\epsilon\rightarrow0^+} I_4(\epsilon) = \lim_{R\rightarrow\infty} I_5(R) = 0\nonumber ,
		\end{equation}
		and an easy computation yields 
		\begin{align*}
			l_{\alpha,\beta}(z)=-\frac{1}{2\pi i}\lim_{R\rightarrow\infty,\,\epsilon\rightarrow0^+} I_2(\epsilon,R) +I_3(R,\epsilon).
		\end{align*}
		Indeed, 
		\begin{align}
			&\frac{1}{2\pi i}\lim_{\epsilon\rightarrow 0^+,R\rightarrow\infty}\left[I_2(\epsilon,R) + I_3(R,\epsilon)\right]\nonumber\\
			&= \frac{1}{2\pi i}\left(\int_{0}^{\infty} \frac{e^{-t} t^{\alpha-\beta}e^{(\alpha-\beta)\pi i}}{t^\alpha e^{\alpha\pi i} +z} dt  - \int_{0}^{\infty} \frac{e^{-t} t^{\alpha-\beta}e^{-(\alpha-\beta)\pi i}}{t^\alpha e^{-\alpha\pi i} + z} dt\right) \nonumber\\
			& = -\frac{1}{2\pi i}2i \int_{0}^{\infty} e^{-t} t^{\alpha-\beta}\frac{z\sin[(\alpha-\beta)\pi] -t^\alpha\sin(\beta\pi) }{t^{2\alpha}+2t^\alpha z\cos(\alpha\pi) + z^2} dt\label{E-dec1}
			\\
			&=-\int_0^\infty \mathcal{H}_{\alpha,\beta}(s)\exp{(- z^{\frac{1}{\alpha}}s^{\frac{1}{\alpha}})}z^{\frac{1}{\alpha}(1-\beta)}ds\label{E-dec2}, \\
			& = -l_{\alpha,\beta}(z) \nonumber,
		\end{align}
		where the change $t\mapsto z^{\frac{1}{\alpha}}s^{\frac{1}{\alpha}}$ was used from (\ref{E-dec1}) to (\ref{E-dec2}). Also, we obtain
		\begin{equation*}
			Res(\Phi,a_\alpha(z)) = \frac{z^{\frac{1-\beta}{\alpha}}}{\alpha} \exp(a_\alpha(z)+\pi i(1-\beta)/\alpha)
		\end{equation*}
		and
		\begin{equation*}
			Res(\Phi,b_\alpha(z)) = \frac{z^{\frac{1-\beta}{\alpha}}}{\alpha} \exp(b_\alpha(z) -\pi i(1-\beta)/\alpha).
		\end{equation*}
		These give us the desired decomposition.
	\end{proof}
	
	In particular, for $z=\vert\xi\vert^2$, $\beta=1$ and $\beta=2$ in Proposition \ref{dec-gen-E}, we earn the decompositions given in \cite[Lemma 1.1]{YF},
	\begin{align}\label{key1}
		E_{\alpha,1}(-z)=\omega_{\alpha,1}(z)+l_{\alpha,1}(z)
	\end{align}
	and in \cite[Lemma 1.2-(IV)]{YF1},
	\begin{align}\label{key2}
		E_{\alpha,2}(-z)=\omega_{\alpha,2}(z)+l_{\alpha,2}(z).
	\end{align}
	Notice that $\omega_{\alpha,1}(z)$ oscillates with frequency $\sin(\pi/\alpha)$ and  amplitude decaying exponentially with rate $\vert \cos(\pi/\alpha)\vert$, in view of
	\begin{align}
		\omega_{\alpha,1}(z)=\frac{2}{\alpha}\exp (z^{\frac{1}{\alpha}}\cos(\pi/\alpha))\cos(z^{\frac{1}{\alpha}}\sin (\pi/\alpha))\nonumber.
	\end{align}
	On the other hand, the function $l_{\alpha,1}(z)$ exhibits the relaxation phenomena of $E_{\alpha,1}(-z)$, namely,
	\begin{equation}
		l_{\alpha,1}(z)=\int_{0}^{\infty}\mathcal{H}_{\alpha,1}(s)\exp(-s^{\frac{1}{\alpha}}z^{\frac{1}{\alpha}})ds=\int_{0}^{\infty}\exp(-s^{\frac{1}{\alpha}}z^{\frac{1}{\alpha}})d\mu_{\alpha}(s),\nonumber
	\end{equation}
	where  
	\begin{equation}
		\mathcal{H}_{\alpha,1}(s)=\frac{\sin (\alpha \pi)}{\alpha\pi}\frac{1}{s^{2}+2s\cos(\alpha\pi)+1}
	\end{equation}
	and  $d\mu_{\alpha}(s)=\mathcal{H}_{\alpha,1}(s)ds$ is a finite measure in $\mathbb{R}_{+}$ such that $\mu_{\alpha}(\mathbb{R}_{+})=2-\frac{2}{\alpha}.$ Furthermore, when $\beta=\alpha$, the decomposition (\ref{decomp}) is useful to show that the map $G_{\alpha,\beta}(\cdot)$, $\beta=1,2$, is differentiable for $t>0$. Indeed, see (\ref{cgal1}) and (\ref{cgal2}) below.
	
	\subsection{Mikhlin estimates for $E_{\alpha,\beta}(-\sigma(\xi))$}\label{Mikhlin}
	In the sequence we provide estimates for $E_{\alpha,1}(-\sigma(\xi))$,  $E_{\alpha,2}(-\sigma(\xi))$ and  $E_{\alpha,\alpha}(-\sigma(\xi))$,  where $\sigma \in C^{\infty}(\mathbb{R}^{N}\backslash\{0\} ; (0,\infty))$ is the symbol of the Fourier multiplier $\sigma(D)$
	\begin{align}
		\sigma(D)f=\mathcal{F}^{-1}\sigma(\xi)\mathcal{F}f(\xi),\; f\in\mathcal{S}(\mathbb{R}^N).\nonumber
	\end{align}
	Consider the change $z\mapsto \sigma(\xi)$ into (\ref{key1}) and write it as follows
	\begin{equation}\label{key-decomp1}
		E_{\alpha,1}(-\sigma(\xi))=
		\omega_{\alpha,1}(\sigma(\xi))+l_{\alpha,1}(\sigma(\xi)).
	\end{equation}
	
	\begin{proposition}\label{fund-lemma} 
		Let $\sigma(\xi)\in C^{\infty}(\mathbb{R}^{N}\backslash\{0\})$ be a function homogeneous of degree $k>0$ and such that
		\begin{equation}\label{key-sym}
			\left\vert \frac{\partial^{\gamma}}{\partial \xi^{\gamma}}\left[\sigma(\xi)\right] \right\vert \leq
			A\left\vert \xi\right\vert ^{k-\left\vert \gamma\right\vert },\;\text{ } \xi\neq0
		\end{equation} 
		for all multi-index $\gamma\in(\mathbb{N}\cup\{0\})^{N}$ with $\left\vert \gamma\right\vert
		\leq\lbrack N/2]+1$. Let $1<\alpha<2$ and $0\leq\delta<k$, there exists $C>0 $ such that
		\begin{equation}
			\left\vert \frac{\partial^{\gamma}}{\partial \xi^{\gamma}}\left[ \left\vert
			\xi\right\vert ^{\delta}E_{\alpha,1}(-\sigma(\xi))\right] \right\vert \leq
			CA\left\vert \xi\right\vert ^{-\left\vert \gamma\right\vert },\;\text{ } \xi\neq0.
			\label{point1}
		\end{equation}
	\end{proposition}
	
	\begin{proof} Taking into account (\ref{key-sym}) we obtain
		\begin{align}
			\left\vert \frac{\partial^{\gamma} }{\partial \xi^{\gamma}} [\sigma(\xi)]^{l}\right\vert\leq CA\, \vert\xi\vert^{-\vert\gamma\vert}\vert \xi\vert^{k l}, \text{ for all } l\in  \mathbb{R}.\label{aux-deriv-symbol}
		\end{align}
		Hence, the $\gamma^{th}-$order derivative of the parcel $\vert\xi\vert^{\delta}\omega_{\alpha,1}(\sigma(\xi))$ can be expressed 
		\begin{align}
			\left\vert\frac{\partial^{\gamma} }{\partial \xi^{\gamma}}\left[\vert\xi\vert^{\delta}\omega_{\alpha,1}(\sigma(\xi))\right]\right\vert&=\left\vert\frac{\partial^{\gamma} }{\partial \xi^{\gamma}}\left[\vert\xi\vert^{\delta}\exp (e^{\frac{i\pi}{\alpha}}\sigma(\xi)^{\frac{1}{\alpha}}) +\vert\xi\vert^{\delta}\exp (e^{-\frac{i\pi}{\alpha}}\sigma(\xi)^{\frac{1}{\alpha}})\right]\right\vert\nonumber\\
			&\leq C\vert\xi\vert^{-\vert\gamma\vert}\left[c_0\vert\xi\vert^{\delta}+c_1\vert\xi\vert^{\delta+\frac{k}{\alpha}}+\cdots +c_{\vert\gamma\vert}\vert\xi\vert^{\delta+\frac{\vert\gamma\vert k}{\alpha}}\right] e^{\cos(\frac{\pi}{\alpha})\sigma(\xi)^{\frac{1}{\alpha}}}\nonumber\\
			&\leq CA\vert\xi\vert^{-\vert\gamma\vert}. \label{w-est}
		\end{align}
		To estimate the parcel $l_{\alpha,1}(\sigma(\xi))$, recall that
		\begin{align}
			l_{\alpha,1}(\sigma(\xi))=\int_0^{\infty}\mathcal{H}_{\alpha,1}(s)\exp(-s^{\frac{1}{\alpha}}\sigma(\xi)^{\frac{1}{\alpha}})ds.\nonumber
		\end{align}
		Using the homogeneity $\sigma(\lambda \xi)=\lambda^{k}\sigma(\xi)$, we estimate
		\begin{align}
			&\left\vert\frac{\partial^{\gamma}}{\partial\xi^{\gamma}}\left[\left\vert \xi\right\vert^{\delta}e^{-s^{\frac{1}{\alpha}}\sigma(\xi)^{\frac{1}{\alpha}}}\right]\right\vert
			\leq C \vert\xi\vert^{-\vert\gamma\vert}\left[c_0\vert\xi\vert^{\delta}+c_1\vert\xi\vert^{\delta+\frac{k}{\alpha}}s^{\frac{1}{\alpha}}+\cdots +c_{\vert\gamma\vert}\vert\xi\vert^{\delta+\frac{\vert\gamma\vert k}{\alpha}}s^{\frac{\vert\gamma\vert}{\alpha}}\right] e^{-s^{\frac{1}{\alpha}}\sigma(\xi)^{\frac{1}{\alpha}}}\nonumber\\
			&=C s^{-\frac{\delta}{k}}\vert\xi\vert^{-\vert\gamma\vert}\left[c_0\left\vert s^{\frac{1}{k}}\xi\right\vert^{\delta}+c_1\left\vert s^{\frac{1}{k}}\xi\right\vert^{\delta+\frac{k}{\alpha}}+\cdots +c_{\vert\gamma\vert}\left\vert s^{\frac{1}{k}}\xi\right\vert^{\delta+\frac{\vert\gamma\vert k}{\alpha}}\right] e^{-[\sigma(s^{\frac{1}{k}}\xi)]^{\frac{1}{\alpha}}} \label{aux-est4}\\
			&\leq CAs^{-\frac{\delta}{k}}\vert\xi\vert^{-\vert\gamma\vert} .\nonumber
		\end{align}
		Then,
		\begin{align}
			\left\vert\frac{\partial^{\gamma}}{\partial\xi^{\gamma}}\left[\vert\xi\vert^{\delta}l_{\alpha,1}(\sigma(\xi))\right]\right\vert&=\frac{\sin (\alpha \pi)}{\alpha\pi}\int_{0}^{\infty}\frac{1}{s^{2}+2s\cos(\alpha\pi)+1}\left\vert\frac{\partial^{\gamma}}{\partial\xi^{\gamma}}\left[\vert\xi\vert^{\delta}e^{ -s^{\frac{1}{\alpha}}\sigma(\xi)^{\frac{1}{\alpha}}}\right]\right\vert ds\nonumber\\
			&\leq CA\left(\frac{\sin (\alpha \pi)}{\alpha\pi}\int_{0}^{\infty}\frac{s^{-\frac{\delta}{k}}}{s^{2}+2s\cos(\alpha\pi)+1}ds\right) \vert\xi\vert^{-\vert\gamma\vert}\nonumber\\
			&\leq CA\, \vert\xi\vert^{-\vert\gamma\vert} ,\nonumber
		\end{align}
		because $0\leq \delta<k$. These estimates prove the proposition.
	\end{proof}
	
	\smallskip
	In general, we obtain the following proposition for the two-parametric Mittag-Leffler function.
	\begin{proposition}\label{fund-lemma2} 
		Let $\sigma(\xi)\in C^{\infty}(\mathbb{R}^{N}\backslash\{0\})$ be a homogeneous function of degree $k>0$ and satisfying (\ref{key-sym}), for all multi-index $\gamma\in(\mathbb{N}\cup\{0\})^{N}$ with $\left\vert \gamma\right\vert
		\leq\lbrack N/2]+1$.  Then, there exists a positive constant $C$ (independent of $\delta$ and $k$)  such that 
		\begin{equation}
			\left\vert \frac{\partial^{\gamma}}{\partial \xi^{\gamma}}\left[ \left\vert
			\xi\right\vert ^{\delta}E_{\alpha,\beta}(-\sigma(\xi))\right] \right\vert \leq
			CA\left\vert \xi\right\vert ^{-\left\vert \gamma\right\vert },\;\text{ } \xi\neq0
			\label{point2}
		\end{equation}
		provided that $1<\alpha<2$ and  $\,k\left(\frac{\beta}{\alpha}-\frac{1}{\alpha}\right)\leq\delta<k$.
	\end{proposition}
	
	\begin{proof} The proof is similar the proof of Proposition \ref{fund-lemma}. Indeed, proceeding as in (\ref{w-est}), it follows that
		\begin{align}
			\left\vert \frac{\partial^{\gamma_1}}{\partial \xi^{\gamma_1}} h_{\alpha,\beta}(\xi)\right\vert 
			&:=\left\vert\frac{\partial^{\gamma_1} }{\partial \xi^{\gamma_1}}\left[\vert\xi\vert^{\delta}\exp \left(a_\alpha(\sigma(\xi))+\frac{1-\beta}{\alpha}\pi i\right) +\vert\xi\vert^{\delta}\exp \left(b_\alpha(\sigma(\xi))-\frac{1-\beta}{\alpha}\pi i\right)\right]\,\right\vert \nonumber\\
			&\leq CA\vert\xi\vert^{-\vert\gamma_1\vert}\left[c_0\vert\xi\vert^{\delta}+c_1\vert\xi\vert^{\delta+\frac{k}{\alpha}}+\cdots +c_{\vert\gamma_1\vert}\vert\xi\vert^{\delta+\frac{\vert\gamma_1\vert k}{\alpha}}\right] e^{\cos(\pi/\alpha) \sigma( \xi )^ \frac{1}{\alpha}},\label{halpha}
		\end{align}
		for all multi-index $\gamma_1$. Hence, Leibniz's rule, (\ref{aux-deriv-symbol}) and (\ref{halpha}) give us 
		\begin{align}
			& \left\vert\frac{\partial^{\gamma_1}}{\partial \xi^{\gamma}}\left[\vert\xi\vert^{\delta}\omega_{\alpha,\beta}(\sigma(\xi))\right]\right\vert  \leq  \sum_{\gamma_1\leq \gamma}\binom{\gamma}{\gamma_1}\left\vert \frac{\partial^{\gamma_1}}{\partial \xi^{\gamma_1}}[\sigma(\xi)]^{\frac{1-\beta}{\alpha}}\right\vert\,\left\vert \frac{\partial^{\gamma-\gamma_1}}{\partial \xi^{\gamma-\gamma_1}} \left[h_{\alpha,\beta}(\xi)\right]\right\vert\nonumber\\
			&\leq CA\vert\xi\vert^{-\vert\gamma_1\vert -\vert\gamma-\gamma_1\vert}\left[c_0\vert\xi\vert^{\delta+k\left(\frac{1}{\alpha}-\frac{\beta}{\alpha}\right)}+\cdots +c\vert\xi\vert^{\delta+k\left(\frac{1}{\alpha}-\frac{\beta}{\alpha}\right)+\frac{\vert\gamma-\gamma_1\vert k}{\alpha}}\right] e^{\cos(\pi/\alpha) \sigma(\xi)^ \frac{1}{\alpha}}\nonumber\\
			&\leq C A \vert\xi\vert^{-\vert\gamma\vert},\nonumber
		\end{align}
		in view of $\delta+k\left(\frac{1}{\alpha}-\frac{\beta}{\alpha}\right)\geq0$. Also, using (\ref{aux-deriv-symbol}) and (\ref{aux-est4}), the Leibniz's rule yields
		\begin{align}
			\left\vert \frac{\partial^{\gamma}}{\partial \xi^{\gamma}} \left[\left(\sigma(\xi)^{\frac{1-\beta}{\alpha}}\right)\left(\vert\xi\vert^{\delta}\exp (-s^{\frac{1}{\alpha}}\sigma(\xi)^{\frac{1}{\alpha}})\right)\right]\right\vert\nonumber 
			\leq C A\vert\xi\vert^{-\vert\gamma\vert}s^{-\frac{\delta}{k}+\frac{\beta}{\alpha}-\frac{1}{\alpha}}.\nonumber
		\end{align}
		
		Hence, we estimate
		\begin{align}
			\left\vert\frac{\partial^{\gamma}}{\partial \xi^{\gamma}} \left[\vert\xi\vert^{\delta}l_{\alpha,\beta}(\sigma(\xi))\right] \right\vert &\leq 
			\int_{0}^{\infty} \vert \mathcal{H}_{\alpha,\beta}(s) \vert \left\vert \frac{\partial^{\gamma}}{\partial \xi^{\gamma}} \left[\sigma(\xi)^{\frac{1-\beta}{\alpha}}\vert\xi\vert^{\delta}\exp (-s^{\frac{1}{\alpha}}\sigma(\xi)^{\frac{1}{\alpha}})\right]\right\vert ds\nonumber\\
			&= CA\,(I +II)\vert\xi\vert^{-\vert\gamma\vert},\nonumber\\
			&\leq C \vert\xi\vert^{-\vert\gamma\vert},\nonumber
		\end{align}
		where the  integrals $I$ and $II$ are defined by (see (\ref{axi2}))
		\begin{equation}
			I=\frac{\sin [(\alpha -\beta) \pi]}{\alpha\pi}\int_{0}^{\infty}\frac{s^{-\frac{\delta}{k}}}{s^{2}+2s\cos(\alpha\pi)+1} ds
		\end{equation}
		and
		\begin{equation}
			II=-\frac{\sin (\beta \pi)}{\alpha\pi}\int_{0}^{\infty}\frac{s^{1-\frac{\delta}{k}}}{s^{2}+2s\cos(\alpha\pi)+1} ds.
		\end{equation}
		Those integrals are finite in view of $\delta<k$. This completes the proof of the proposition.
	\end{proof}

\section{Sobolev-Morrey estimates}\label{sme}
	In this section we obtain fundamental estimates which will be important to prove  Theorem \ref{gw}.

\subsection{Linear estimates}\label{linear-est}

Here, we present some estimates of the  Mittag-Leffler operators $\left\{G_{\alpha,\beta}(t)\right\}_{t\geq 0}$ in  Sobolev-Morrey spaces. Indeed, based on Propositions \ref{fund-lemma} and \ref{fund-lemma2} with the homogeneous symbol $\sigma(\xi)=4\pi^2\vert\xi\vert^2$ of degree $2$, the following lemma can be proved by proceeding as in \cite[Lemma 3.1-(i)]{MJ}.
	
	\begin{lemma}\label{galpha}
		Let $\gamma_1\leq \gamma_2\in\mathbb{R}$, $ 1<p_1\leq p_2<\infty,\ 0\leq \mu< N$,  $1< \alpha<2$ and $\lambda=(\gamma_2-\gamma_1)+\frac{N-\mu}{p_1}-\frac{N-\mu}{p_2}$. There is a constant $C$ such that
		\begin{align}
			\| G_{\alpha,1}(t) f\|_{\mathcal{M}_{p_2,\mu}^{\gamma_2}} &\leq Ct^{-\frac{\alpha}{2}\lambda}\|f\|_{\mathcal{M}_{p_1,\mu}^{\gamma_1}},\; \text{ if }\;\lambda<2,\label{item-i}\\
			\| G_{\alpha,2}(t) f\|_{\mathcal{M}_{p_2,\mu}^{\gamma_2}} &\leq Ct^{-\frac{\alpha}{2}\lambda} \,\|f\|_{\mathcal{M}_{p_1,\mu}^{\gamma_1-\frac{2}{\alpha}}},\; \text{ if } \;\lambda +\frac{2}{\alpha}<2,\label{item-ii}  \\
			\| G_{\alpha,\alpha}(t) f\|_{\mathcal{M}_{p_2,\mu}^{\gamma_2}} &\leq Ct^{\alpha-1-\frac{\alpha}{2}\lambda}\|f\|_{\mathcal{M}_{p_1,\mu}^{\gamma_1}},\; \text{ if }\;\left (2-\frac{2}{\alpha}\right)<\lambda<2,\label{item-iii}
		\end{align}
		for all $f\in\mathcal{S}'(\mathbb{R}^N)$.
	\end{lemma}

	We finish this subsection by noticing that $\left\{\partial_tG_{\alpha,1}(t)\right\}_{t\geq 0}$ and $\left\{\partial_tG_{\alpha,2}(t)\right\}_{t\geq 0}$ are bounded in Morrey spaces. Indeed, a straightforward computation gives us 
	\begin{equation}
		\frac{d}{dt} E_{\alpha,1}(-4\pi^2\vert\xi\vert^2t^{\alpha}) = -4\pi^2\vert\xi\vert^2\left[t^{\alpha-1}E_{\alpha,\alpha}(-4\pi^2\vert\xi\vert^2t^{\alpha})\right],\, t>0 \text{ and } \xi\neq 0.\nonumber
	\end{equation}
	It follows from Lemma \ref{galpha}-(iii) that 
	\begin{equation}\label{cgal1}
		\Vert  \partial_tG_{\alpha,1}(t)f\Vert_{\mathcal{M}_{p_2,\mu}}\leq C\Vert  G_{\alpha,\alpha}(t)f\Vert_{\mathcal{M}_{p_2,\mu}^{2}}\leq C t^{-\frac{\alpha}{2}\left(\frac{N-\mu}{p_1}-\frac{N-\mu}{p_2}\right)-1}\Vert  f\Vert_{\mathcal{M}_{p_1,\mu}}\nonumber .
	\end{equation}
	Using 
	\begin{equation}
	tE_{\alpha,2}(-4\pi^2\vert\xi\vert^2t^\alpha)=\int_{0}^{t}E_{\alpha,1}(-4\pi^2\vert\xi\vert^2s^{\alpha}) ds ,
	\end{equation}
	Lemma \ref{galpha}-(i) yields
	\begin{equation}\label{cgal2}
		\Vert  \partial_tG_{\alpha,2}(t)f\Vert_{\mathcal{M}_{p_2,\mu}}=\Vert  G_{\alpha,1}(t)f\Vert_{\mathcal{M}_{p_2,\mu}}\leq C t^{-\frac{\alpha}{2}\left(\frac{N-\mu}{p_1}-\frac{N-\mu}{p_2}\right)}\Vert  f\Vert_{\mathcal{M}_{p_1,\mu}}.\nonumber
	\end{equation}
	
	\subsection{Nonlinear estimates}\label{non_estmate}
	This subsection is devoted to estimate the nonlinear term $\mathcal{N}_{\alpha}(u)$ on the functional space $X_{\beta}$. Firstly, let us denote $\mathbf{B}(\nu,\eta)$ by \textit{special beta function} $\mathbf{B}(\nu,\eta)=\int_0^1(1-t)^{\nu-1}t^{\eta-1}dt$ which is finite, for all $\eta,\nu>0$. Let $k_1,k_2,k_3<1$, for  $t>0$ and $s>0$ the changes of variable $\tau \mapsto \tau s$ and $s\mapsto st$ give us 
	\begin{align}
		I(t)&=\int_0^t(t-s)^{-k_1}\int_0^s(s-\tau)^{-k_2}\tau^{-k_3}d\tau ds=\mathbf{B}(1-k_2,1-k_3)\int_0^t(t-s)^{-k_1}s^{-k_2-k_3+1}ds\nonumber\\
		&=\mathbf{B}(1-k_2,1-k_3)\mathbf{B}(1-k_1, 2-k_2-k_3)t^{2-k_1-k_2-k_3}.\label{beta}
	\end{align}
	We freely make use of (\ref{beta}) in the next proof.
	
	\begin{lemma}
		Under assumptions of Theorem \ref{gw}, there is a positive constant $K=K(\kappa_1,\kappa_2)$ such that
		\begin{align} \label{est-Non}
			\left\|\mathcal{N}_{\alpha}(u)-\mathcal{N}_{\alpha}(v)\right\|_{X_{\beta}}\leq K\Vert u-v\Vert_{X_{\beta}}\left[\Vert u\Vert_{X_{\beta}}^{\rho-1}+\Vert v\Vert_{X_{\beta}}^{\rho-1}+\Vert u\Vert_{X_{\beta}}^{q-1}+\Vert v\Vert_{X_{\beta}}^{q-1}\right] .
		\end{align}
	\end{lemma}
	
	\begin{proof} Recall $\mathcal{N}_{\alpha}(u)$ and rewrite it as follows
		\begin{align}
			\mathcal{N}_{\alpha}(u)(t)&=\int_{0}^{t}G_{\alpha,1}(t-s)\int_{0}^{s}r_{\alpha}(s-\tau) \left(\kappa_2\vert u\vert^{\rho-1} u+\kappa_1\vert \nabla_x u\vert^{q}\right) d\tau ds \\
			& =: \mathcal{N}_{\alpha}^1(u)(t)+\mathcal{N}_{\alpha}^2(u)(t).\nonumber
		\end{align}
		The proof is divided in three steps.
		
		\noindent \textbf{First step:} Estimates for $\mathcal{N}_{\alpha}^1(u)$. In the inequality (\ref{item-i}), let $(\gamma_1,\gamma_2,p_1,p_2)=(0,1,r/\rho,r)$ and  $1<\rho<r$ to estimate
		\begin{align}
			\Vert \mathcal{N}_{\alpha}^1(u)(t)-\mathcal{N}_{\alpha}^1(v)(t) \Vert_{\mathcal{M}^1_{r,\mu}}&\leq C \int_0^t (t-s)^{-\lambda_1}\int_0^sr_{\alpha}(s-\tau)\left\Vert\,f(u)-f(v)\,\right\Vert_{\mathcal{M}_{r/\rho,\mu}}d\tau ds\nonumber
		\end{align}
		where $f(u)(\tau)=\kappa_1\vert u(\tau)\vert^{\rho-1}u(\tau)$ and $\lambda_1=\frac{\alpha}{2}+\frac{\alpha}{2}\left(\frac{N-\mu}{r/\rho}-\frac{N-\mu}{r}\right)$. Using the facts  
		\begin{equation}
			\vert\, \vert a\vert^{\rho-1} a-\vert b\vert^{\rho-1}b\vert\leq C\vert a-b\vert\left(\vert a\vert^{\rho-1}+ \vert b\vert^{\rho-1}\right), \text{ for all }  \rho>1 \label{iq-fund}
		\end{equation}
		and $\frac{\rho}{r}=\frac{1}{r}+\frac{\rho-1}{r}$,  the H\"older inequality (\ref{eq:holder}) yields
		\begin{align}
			\Vert \mathcal{N}_{\alpha}^1(u)(t)-\mathcal{N}_{\alpha}^1(v)(t) \Vert_{\mathcal{M}^1_{r,\mu}}&\leq C\vert\kappa_2\vert \int_0^t (t-s)^{-\lambda_1}\theta(s)ds\label{B_1} ,
		\end{align}
		where $\theta(s)$ is given by
		\begin{align}
			\theta(s)&=\int_0^s(s-\tau)^{\alpha-2}\Vert u(\tau)-v(\tau)\Vert_{\mathcal{M}_{r,\mu}}\left(\Vert u(\tau)\Vert_{\mathcal{M}_{r,\mu}}^{\rho-1}+\Vert v(\tau)\Vert_{\mathcal{M}_{r,\mu}}^{\rho-1}\right)d\tau\nonumber\\
			&\leq C\int_0^s(s-\tau)^{\alpha-2}\tau^{-\rho\beta}\,\tau^{\beta}\Vert u(\tau)-v(\tau)\Vert_{\mathcal{M}_{r,\mu}} \times \nonumber\\
			&\times\tau^{\beta(\rho-1)}\left(\Vert u(\tau)\Vert_{\mathcal{M}_{r,\mu}}^{\rho-1}+\Vert v(\tau)\Vert_{\mathcal{M}_{r,\mu}}^{\rho-1}\right)d\tau\nonumber\\
			&\leq C\int_0^s(s-\tau)^{\alpha-2}\tau ^{-\rho\beta}d\tau\, \Vert u-v\Vert_{X_{\beta}}\left(\Vert u\Vert_{X_{\beta}}^{\rho-1}+\Vert v\Vert_{X_{\beta} }^{\rho-1}\right).\label{theta1}
		\end{align}
		Notice that $\alpha (\rho-1)\frac{N-\mu}{2r}=\alpha -(\rho-1)\beta$ yield to
		\begin{align}
			-\lambda_1+\alpha-\rho\beta&=-\frac{\alpha}{2}+(\rho -1)\beta -\alpha +\alpha- \rho\beta=-\frac{\alpha}{2}-\beta.\nonumber
		\end{align}
		It follows that (\ref{B_1}) can be bounded by
		\begin{align}
			\Vert \mathcal{N}_{\alpha}^1(u)(t)-\mathcal{N}_{\alpha}^1(v)(t)\Vert_{\mathcal{M}^1_{r,\mu}}&\leq C\vert\kappa_2\vert I_1(t) \,\Vert u-v\Vert_{X_{\beta}}\left(\Vert u\Vert_{X_{\beta}}^{\rho-1}+\Vert v\Vert_{X_{\beta} }^{\rho-1}\right)\nonumber\\
			&\leq C\vert\kappa_2\vert t^{-\beta-\frac{\alpha}{2}} \,\Vert u-v\Vert_{X_{\beta}}\left(\Vert u\Vert_{X_{\beta}}^{\rho-1}+\Vert v\Vert_{X_{\beta} }^{\rho-1}\right),
		\end{align}
		where the integral $I_1(t)$ (see (\ref{beta})) satisfies
		\begin{align}
			I_1(t)&=\int_0^t (t-s)^{-\lambda_1}\left(\int_0^s(s-\tau)^{\alpha-2}\tau ^{-\rho\beta}d\tau\right) ds=C t^{-\lambda_1+\alpha -\rho\beta}=Ct^{-\beta -\frac{\alpha}{2}} .\nonumber
		\end{align}
		
		Proceeding in a similar fashion, we obtain
		\begin{align}
			\Vert\mathcal{N}_{\alpha}^1(u)(t)-\mathcal{N}_{\alpha}^1(v)(t)\Vert_{\mathcal{M}_{r,\mu}}&\leq C\vert\kappa_2\vert  J_{1}(t) \,\Vert u-v\Vert_{X_{\beta}}\left(\Vert u\Vert_{X_{\beta}}^{\rho-1}+\Vert v\Vert_{X_{\beta} }^{\rho-1}\right)\nonumber\\
			&\leq C\vert\kappa_2\vert  t^{-\beta} \,\Vert u-v\Vert_{X_{\beta}}\left(\Vert u\Vert_{X_{\beta}}^{\rho-1}+\Vert v\Vert_{X_{\beta}}^{\rho-1}\right),
		\end{align}
		where $J_1(t)$ is given by 
		\begin{align}
			J_1(t)&=\int_0^t (t-s)^{-\vartheta_{1}}\int_0^s(s-\tau)^{\alpha-2}\tau ^{-\rho\beta}d\tau ds,
		\end{align}
		and  $\vartheta_1=\frac{\alpha}{2}\left(\frac{N-\mu}{r/\rho}-\frac{N-\mu}{r}\right)$. In view of
		\begin{align}
			-\vartheta_1+\alpha-\rho\beta&=\frac{\alpha}{2}\left(\frac{N-\mu}{r/\rho}-\frac{N-\mu}{r}\right)+\alpha-\rho\beta=(\rho -1)\beta -\alpha +\alpha- \rho\beta=-\beta\nonumber ,
		\end{align}
		one has
		\begin{align}
			J_1(t)&=Ct^{-\vartheta_1+\alpha-\rho\beta}=Ct^{-\beta}.
		\end{align}
		The convergence of $I_{1}(t)$ and $J_{1}(t)$ follows from (\ref{param-hip2}). In fact, $\lambda_1= \frac{\alpha}{2}+\vartheta_1<1$ because $\frac{p}{r}<\left(\frac{1}{\alpha}-\frac{1}{2}\right)$ is equivalent to $\vartheta_1=\alpha\frac{p}{r}<1-\frac{\alpha}{2}$. And $\left(1-\frac{p}{r}\right) < \frac{\rho-1}{\alpha}\left(\frac{1}{q}-\frac{\alpha}{2}\right)=\frac{\rho-1}{\alpha\rho}\left(\frac{\rho+1}{2}-\frac{\alpha\rho}{2}\right)<\frac{\rho-1}{\alpha\rho}$ leads to $\rho\beta<1$.

		\noindent \textbf{Second step:} Estimates for $\mathcal{N}_{\alpha}^2(u)$. Using inequality (\ref{item-i}) with $(\gamma_1,\gamma_2,p_1,p_2)=(0,1,r/q,r)$, in view of $1<q<\rho<r$, we obtain
		\begin{align}
			\Vert\mathcal{N}_{\alpha}^2(u)(t)-\mathcal{N}_{\alpha}^2(v)(t)\Vert_{\mathcal{M}_{r,\mu}^1}&\leq C \int_0^t (t-s)^{-\lambda_2}\int_0^sr_{\alpha}(s-\tau)\left\Vert\,g(u)-g(v)\,\right\Vert_{\mathcal{M}_{r/q,\mu}}d\tau ds\nonumber ,
		\end{align}
		where $g(f)(\tau)=\kappa_2\vert \nabla_xf(\tau)\vert^q$ and $\lambda_2=\frac{\alpha}{2}+\frac{\alpha}{2}\left(q\frac{N-\mu}{r}-\frac{N-\mu}{r}\right)$. It follows from inequality
		\begin{align}
			\Vert\,\vert\nabla_x u(t)\vert^q-\vert\nabla_x v(t)\vert^q\,\Vert_{\mathcal{M}_{r,\mu}}\leq C  \Vert u(t)- v(t)\Vert_{\mathcal{M}_{r,\mu}^1}\left(\Vert  u(t)\Vert_{\mathcal{M}_{r,\mu}^1}^{q-1}+\Vert v(t)\Vert_{\mathcal{M}_{r,\mu}^1}^{q-1}\right)\nonumber
		\end{align}
		that
		\begin{align}
			\Vert\mathcal{N}_{\alpha}^2(u)(t)-\mathcal{N}_{\alpha}^2(v)(t)\Vert_{\mathcal{M}_{r,\mu}^1}&\leq C\vert\kappa_1\vert \int_0^t (t-s)^{-\lambda_2}\tilde{\theta}(s)ds\label{B_2},
		\end{align}
		where $\tilde{\theta}(s)$ is bounded by
		\begin{align}
			\tilde{\theta}(s)&= C\int_0^s(s-\tau)^{\alpha-2}\Vert u(\tau)-v(\tau)\Vert_{\mathcal{M}_{r,\mu}^1}\left(\Vert u(\tau)\Vert_{\mathcal{M}_{r,\mu}^1}^{q-1}+\Vert v(\tau)\Vert_{\mathcal{M}_{r,\mu}^1 }^{q-1}\right)d\tau\nonumber\\
			&\leq C\int_0^s(s-\tau)^{\alpha-2}\tau ^{-q(\beta +\frac{\alpha}{2})}d\tau\, \Vert u-v\Vert_{X_{\beta}}\left(\Vert u\Vert_{X_{\beta}}^{q-1}+\Vert v\Vert_{X_{\beta} }^{q-1}\right).\label{theta2}
		\end{align}
		Hence, we estimate 
		\begin{align}
			\left\Vert\mathcal{N}_{\alpha}^2(u)(t)-\mathcal{N}_{\alpha}^2(v)(t)\right\Vert_{\mathcal{M}_{r,\mu}^1}&\leq C\vert \kappa_1\vert I_2(t) \,\Vert u-v\Vert_{X_{\beta}}\left(\Vert u\Vert_{X_{\beta}}^{q-1}+\Vert v\Vert_{X_{\beta}}^{q-1}\right)\nonumber\\
			&\leq C\vert\kappa_1\vert t^{-\beta-\frac{\alpha}{2}}\,\Vert u-v\Vert_{X_{\beta}}\left(\Vert u\Vert_{X_{\beta}}^{q-1}+\Vert v\Vert_{X_{\beta}}^{q-1}\right),
		\end{align}
		where $I_2(t)$ is given by
		\begin{align}
			I_2(t)&=\int_0^t (t-s)^{-\lambda_2}\int_0^s(s-\tau)^{\alpha-2}\tau ^{-q(\beta +\frac{\alpha}{2})}d\tau ds=Ct^{-\lambda_2+\alpha-q(\beta+\alpha/2)}=Ct^{-\beta -\frac{\alpha}{2}}.\nonumber
		\end{align}
		Indeed, in view of $q=\frac{2\rho}{\rho+1}$ and  
		\begin{equation}
			(q-1)\beta=\alpha \frac{q-1}{\rho-1}-\alpha (q-1)\frac{N-\mu}{2r}=\frac{\alpha}{2}(2-q)-\alpha (q-1)\frac{N-\mu}{2r},\label{eq-auxN}
		\end{equation}
		we obtain
		\begin{align}
			-\lambda_2+\alpha-q(\beta+\alpha/2)&=-\alpha (q-1)\frac{N-\mu}{2r}-\frac{\alpha}{2} + \alpha-q(\beta+\alpha/2)\nonumber\\
			&=(q-1)\beta-\frac{\alpha}{2}(2-q)-\frac{\alpha}{2} + \alpha-q(\beta+\alpha/2)\nonumber\\
			&=-\frac{\alpha}{2}-\beta.\nonumber
		\end{align}
		It remains to get estimates for  $\sup_{0<t<T}t^{\beta}\Vert\mathcal{N}_{\alpha}^2(u)(t)-\mathcal{N}_{\alpha}^2(v)(t)\Vert_{\mathcal{M}_{r,\mu}}$. Proceeding as before, one has 
		\begin{align}
			\Vert\mathcal{N}_{\alpha}^2(u)(t)-\mathcal{N}_{\alpha}^2(v)(t)\Vert_{\mathcal{M}_{r,\mu}}
			&\leq C\vert\kappa_1\vert J_{2}(t)\,\sup_{0<t<T}t^{\beta +\frac{\alpha}{2}}\Vert \nabla_x u(t)-\nabla_xv(t)\Vert_{\mathcal{M}_{r,\mu}}\times\nonumber\\
			&\times \sup_{0<t<T}t^{(\beta +\frac{\alpha}{2})(q-1)}\left(\Vert \nabla_x u(t)\Vert_{\mathcal{M}_{r,\mu}}^{q-1}+\Vert \nabla_x v(t)\Vert_{\mathcal{M}_{r,\mu}}^{q-1}\right)\nonumber\\
			&\leq C\vert\kappa_1\vert J_{2}(t)\,\Vert u-v\Vert_{X_{\beta}}\left(\Vert u\Vert_{X_{\beta}}^{q-1}+\Vert v\Vert_{X_{\beta}}^{q-1}\right)\nonumber\\
			&\leq C\vert \kappa_1\vert t^{-\beta}\,\Vert u-v\Vert_{X_{\beta}}\left(\Vert u\Vert_{X_{\beta}}^{q-1}+\Vert v\Vert_{X_{\beta}}^{q-1}\right),
		\end{align}
		where $J_2(t)$ satisfy
		\begin{align*}
			J_2(t)&=\int_0^t (t-s)^{-\vartheta_2}\int_0^s(s-\tau)^{\alpha-2}\tau ^{-q(\beta +\frac{\alpha}{2})}d\tau ds=Ct^{-\vartheta_2+\alpha -q(\beta +\alpha/2)}=Ct^{-\beta}.
		\end{align*}
		In fact, the inequality (\ref{item-i}) with $(\gamma_1,\gamma_2, p_1,p_2)=(0, 0,r/q, r)$ implies that  $\vartheta_2=\frac{\alpha}{2}\left(q\frac{N-\mu}{r}-\frac{N-\mu}{r}\right)=\alpha (q-1)\frac{N-\mu}{2r}$, which by (\ref{eq-auxN}) give us
		\begin{align}
			-\vartheta_2+\alpha -q(\beta+\frac{\alpha}{2}) = (q-1)\beta -\frac{\alpha}{2}(2-q)+\alpha -q(\beta+\frac{\alpha}{2})=-\beta.
		\end{align}
		The convergence of the beta functions $I_2(t), J_2(t)$ follows by our hypotheses in (\ref{param-hip2}), because  $\left(1-\frac{p}{r}\right) < \frac{\rho-1}{\alpha}\left(\frac{1}{q}-\frac{\alpha}{2}\right)$ is equivalent to $q(\beta+\alpha/2)<1$ and $\frac{p}{r}<\left(\frac{1}{\alpha}-\frac{1}{2}\right)$ and $q=\frac{2\rho}{\rho+1}$ yields to $\vartheta_2=\frac{\alpha}{\rho+1}\frac{p}{r}<\left( 1-\frac{\alpha}{2}\right)$ which is equivalent to $ \lambda_2=\frac{\alpha}{2}+\vartheta_2<1$.
		
		\noindent \textbf{Third step:} The two steps above lead to
		\begin{align*}
			&\Vert  \mathcal{N}_{\alpha}(u)-\mathcal{N}_{\alpha}(v)\Vert_{X_{\beta}} \\
			\leq & \Vert  \mathcal{N}_{\alpha}^1(u)-\mathcal{N}_{\alpha}^1(v)\Vert_{X_{\beta}}+\Vert  \mathcal{N}_{\alpha}^2(u)-\mathcal{N}_{\alpha}^2(v)\Vert_{X_{\beta}}\\
			\leq &K(\kappa_1,\kappa_2)\Vert u-v\Vert_{X_{\beta}}\left[\left(\Vert u\Vert_{X_{\beta}}^{\rho-1}+\Vert v\Vert_{X_{\beta}}^{\rho-1}\right)+\left(\Vert u\Vert_{X_{\beta}}^{q-1}+\Vert v\Vert_{X_{\beta}}^{q-1}\right)\right].
		\end{align*}
	\end{proof}

	\section{Proof of the Theorems}\label{proofs}

	Now, we put all estimates of Section \ref{sme} together to prove our theorems.
	\subsection{Proof of Theorem \ref{gw}}
	\noindent\textbf{Part (i):} Notice that
	\begin{equation}
	\|G_{\alpha,1}(t)\varphi\|_{X_{\beta}} + \|G_{\alpha,2}(t)\psi\|_{X_{\beta}}=\Vert \varphi\Vert_{D(\alpha,\beta)}+\Vert \psi\Vert_{\widetilde{D}(\alpha,\beta)}\leq \varepsilon,\label{est-proof4}
	\end{equation} 
	where $\varepsilon>0$ will be chosen such that
	\begin{align}
		\left(2^{\rho}\varepsilon^{\rho-1}+2^{q}\varepsilon^{q-1}\right)<\frac{1}{2K}.\label{rest}
	\end{align}
	Consider the complete metric $d_B(\cdot,\cdot)$ defined by $d_{B}(u,v)=\Vert u-v\Vert_{X_{\beta}}$ on the ball $B_{X_{\beta}}(2\varepsilon)$ and let $\Lambda:B_{X_{\beta}}(2\varepsilon)\rightarrow B_{X_{\beta}}(2\varepsilon)$ be the operator
	\begin{equation}
		\Lambda (u)(t)= G_{\alpha,1}(t)\varphi +G_{\alpha,2}(t)\psi +\mathcal{N}_{\alpha}(u)(t).\nonumber
	\end{equation}
	We would like to show that  $\Lambda(B_{X_{\beta}}(2\varepsilon))\subset B_{X_{\beta}}(2\varepsilon)$ and $\Lambda$ is a contraction on metric space $(B_{X_{\beta}}(2\varepsilon), d_B)$. In fact, 
	the continuity of $G_{\alpha,j}(\cdot):(0,\infty)\to \mathcal{M}_{r,\mu},\ j=1,2,$ proved at the final of subsection \ref{linear-est}, and the regularization property of the convolution imply that $(\Lambda u):(0,\infty)\to \mathcal{M}_{r,\mu}$ is continuous, whenever $u\in X_\beta$. Further, from Lemma \ref{est-Non} we obtain
	\begin{align}
		\Vert \Lambda(u)-\Lambda(v)\Vert_{X_{\beta}}&=\Vert\mathcal{N}_{\alpha}(u)-\mathcal{N}_{\alpha}(v)\Vert_{X_{\beta}}\nonumber\\
		&\leq K\Vert u-v\Vert_{X_{\beta}}\left[\Vert u\Vert_{X_{\beta}}^{\rho-1}+\Vert v\Vert_{X_{\beta}}^{\rho-1}+\Vert u\Vert_{X_{\beta}}^{q-1}+\Vert v\Vert_{X_{\beta}}^{q-1}\right]\nonumber\\
		&\leq K\left(2^{\rho}\varepsilon^{\rho-1}+2^{q}\varepsilon^{q-1}\right)\Vert u-v\Vert_{X_{\beta}} \label{est-proof3} 
	\end{align}
	for all $u,v\in B_{X_{\beta}}(2\varepsilon)$. Now, noting that $\Lambda(0)(t)= G_{\alpha,1}(t)\varphi+G_{\alpha,2}(t)\psi$, the estimates (\ref{est-proof4}), (\ref{est-proof3}) give us
	\begin{align}
		\Vert \Lambda(u)\Vert_{X_{\beta}}&\leq\Vert G_{\alpha,1}(t)\varphi+G_{\alpha,2}(t)\psi
		\Vert_{X_{\beta}} + \left\Vert\mathcal{N}_{\alpha}(u)-\mathcal{N}_{\alpha}(0)\right\Vert_{X_{\beta}}\nonumber\\
		&\leq \varepsilon + K\left(2^{\rho}\varepsilon^{\rho-1}+2^{q}\varepsilon^{q-1}\right)2\varepsilon < 2\varepsilon
	\end{align}
	in view of (\ref{rest}) and provided that $u\in B_{X_{\beta}}(2\varepsilon)$. Hence, $\Lambda(B_{X_{\beta}}(2\varepsilon))\subset B_{X_{\beta}}(2\varepsilon)$ and $\Lambda$ is a contraction in $B_{X_{\beta}}(2\varepsilon)$. It follows by Banach fixed theorem that there is a \textit{mild solution} $u\in X_{\beta}$ for (\ref{heat-wave})-(\ref{initial-data}) which is unique in the ball $B_{X_{\beta}}(2\varepsilon)$.
	
	\noindent\textbf{Part (ii):}
	Let $u$ and $\tilde{u}$ be two \textit{mild solutions} in $B_{X_{\beta}}(2\varepsilon)$,  obtained in the Part (i), subject to initial data $(\varphi,\psi)$ and $(\tilde{\varphi}, \tilde{\psi})$, respectively. Then,
	\begin{align}
		\Vert u-\tilde{u}\Vert_{X_{\beta}}&\leq \Vert G_{\alpha,1}(\varphi-\tilde{\varphi})\Vert_{X_{\beta}} +\Vert G_{\alpha,2}(\psi-\tilde{\psi})\Vert_{X_{\beta}} + \Vert\mathcal{N}_{\alpha}(u)-\mathcal{N}_{\alpha}(\tilde{u})\Vert_{X_{\beta}}\nonumber\\
		&\leq  \Vert \varphi-\tilde{\varphi}\Vert_{D(\alpha,\beta)}+\Vert \psi-\tilde{\psi}\Vert_{\widetilde{D}(\alpha,\beta)}+K\left( 2^\rho\varepsilon^{\rho-1}+2^q\varepsilon^{q-1}\right) \Vert u-\tilde{u}\Vert_{X_{\beta}}\nonumber
	\end{align}
	which yields the Lipschitz continuity 
	\begin{align}
		\Vert u-\tilde{u}\Vert_{X_{\beta}}\leq \frac{1}{\left(1-\frac{1}{4K}\right)}\left(\Vert \varphi-\tilde{\varphi}\Vert_{D(\alpha,\beta)}+\Vert \psi-\tilde{\psi}\Vert_{\widetilde{D}(\alpha,\beta)}\right).
	\end{align}
\qed	
	\subsection{Proof of Theorem \ref{selfsimilarity}}
	
	Let $\delta_{\lambda}f(x)=f(\lambda x)$ and notice that $\widehat{\delta_{\lambda}f}(\xi)=\lambda^{-N}\widehat{f}(\xi/\lambda)$. Recall that $G_{\alpha, 2}(t)\psi$ and $G_{\alpha,1}(t)\varphi$ are given by
	$$G_{\alpha,2}(t)\psi:=t\,k_{\alpha,2}(t,\cdot)\ast \psi \text{ and }G_{\alpha,1}(t)\varphi:=k_{\alpha,1}(t,\cdot)\ast \varphi,$$ 
	where $\widehat{k}_{\alpha,2}(t,\xi)=tE_{\alpha,2}(-4\pi^2t^{\alpha}\vert\xi\vert^2)$ 
	and $\widehat{k}_{\alpha,1}(t,\xi)=E_{\alpha,1}(-4\pi^2t^{\alpha}\vert\xi\vert^2)$. 
	Notice that
	\begin{align}
		[\delta_{\gamma}G_{\alpha,2}(\gamma^{\frac{2}{\alpha}}t)\psi]^{\wedge}(\xi)&=\gamma^{-N} \gamma^{\frac{2}{\alpha}}t\widehat{k}_{\alpha,2}(\gamma^{\frac{2}{\alpha}}t,\xi/\gamma)\widehat{\psi}(\xi/\gamma)\nonumber\\
		&=\gamma^{\frac{2}{\alpha}}[tE_{\alpha,2}(-4\pi^2t^{\alpha}\vert\xi\vert^2)]\gamma^{-N}\widehat{\psi}(\xi/\gamma)=\gamma^{-\frac{2}{\rho-1}}[G_{\alpha,2}(t)\psi]^{\wedge}(\xi),\nonumber
	\end{align}
	in view of $\delta_{\gamma}\psi(x)=\gamma^{-\frac{2}{\rho-1}-\frac{2}{\alpha}}\psi(x)$. Hence,
	\begin{align}
		[G_{\alpha,2}(t)\psi]_{\gamma}=G_{\alpha,2}(t)\psi_{\gamma}\text{ and } [G_{\alpha,1}(t)\varphi]_{\gamma}=G_{\alpha,1}(t)\varphi_{\gamma}.\nonumber
	\end{align}
	We can easily check that $\mathcal{N}_{\alpha}(u_{\gamma})=[\mathcal{N}_{\alpha}(u)]_{\gamma}$, for all $\gamma>0$. Here, we denoted  $[\mathcal{N}_{\alpha}(u)]_{\gamma}(t,x)=\gamma^{\frac{2}{\rho-1}}\mathcal{N}_{\alpha}(u)(\gamma^{\frac{2}{\alpha}}t,\gamma x) $. Therefore,
	\begin{equation}
		u_\gamma(t) = [G_{\alpha,1}(t)\varphi]_\gamma+ [G_{\alpha,2}(t)\psi]_\gamma + [\mathcal{N}_\alpha(u)]_\gamma = G_{\alpha,1}(t)\varphi_\gamma + G_{\alpha,2}(t)\psi_\gamma + \mathcal{N}_\alpha(u_\gamma)\nonumber
	\end{equation}
	is a mild solution $u_\gamma\in X_{\beta}$ of (\ref{heat-wave})-(\ref{initial-data}). From $\Vert u_\gamma\Vert_{X_\beta} = \Vert u \Vert_{X_\beta}$ and the uniqueness proved in Theorem \ref{gw}-(i), we have 
	\begin{align}
		u(t,x)=u_\gamma(t,x), \text{ a.e } x\in\mathbb{R}^N \text{and for all }\gamma,t>0.\nonumber
	\end{align}
	\qed

	\subsection{Proof of Theorem \ref{symmetry}}\label{prof-symmetry}

	Let $G_{\alpha,2}(t)\psi:=k_{\alpha,2}(t,\cdot)\ast \psi \text{ and }G_{\alpha,1}(t)\varphi:=k_{\alpha,1}(t,\cdot)\ast \varphi,$ 
	be as above. For $T\in\mathcal{A}$, we have
	\begin{align}
		k_{\alpha,1}(t, T(x))&=\int_{\mathbb{R}^N} e^{2\pi i \,\langle T(x),\,\xi\rangle} E_{\alpha,1}(-4\pi^2t^{\alpha}|\xi|^{2}))d\xi\nonumber\nonumber\\
		&=\int_{\mathbb{R}^N} e^{2\pi i \,\langle T(x),\,T(\xi)\rangle} E_{\alpha,1}(-4\pi^2t^{\alpha}|T(\xi)|^{2}))\vert det T\vert d\xi\nonumber\\
		&=\int_{\mathbb{R}^N} e^{2\pi i \,\langle x,\,\xi\rangle} E_{\alpha,1}(-4\pi^2t^{\alpha}|\xi|^{2}))d\xi=k_{\alpha,1}(t, x),\nonumber
	\end{align}
	where we used the change of variable $\xi\mapsto T(\xi)$ and the fact that $\vert \det T\vert = 1$. In a  similar fashion one has
	\[k_{\alpha,2}(t, T(x))=k_{\alpha,2}(t, x).\]
	It follows that
	\begin{align*}
		u_1(t, T(x))&:=\int_{\mathbb{R}^N} k_{\alpha,1}(t, T(x)-y)\varphi(y)dy +\int_{\mathbb{R}^N} k_{\alpha,2}(t, T(x)-y)\psi(y)dy\\
		&=\int_{\mathbb{R}^N} k_{\alpha,1}(t, T(x-z))\varphi(Tz)dz +\int_{\mathbb{R}^N} k_{\alpha,2}(t, T(x-z))\psi(Tz)dz\\
		&=-\int_{\mathbb{R}^N} k_{\alpha,1}(t, x-z)\varphi(z)dz -\int_{\mathbb{R}^N} k_{\alpha,2}(t, x-z)\psi(z)dz=u_1(t,x)
	\end{align*}
	when $\varphi$ and $\psi$ are symmetric under action $\mathcal{A}$. Let $\theta(s,x)=\int_{0}^{s}r_{\alpha}(s-t) \kappa_2\vert u(t,x)\vert^{\rho-1} u(t,x)-\kappa_1 \vert \nabla_xu(t,x)\vert^{q}$ and notice that 
	\begin{align*}
		\theta(s,Tx)&=\int_{0}^{s}r_{\alpha}(s-t) \left(\kappa_2\vert u(t, Tx)\vert^{\rho-1} u(t,Tx)-\kappa_1\vert \nabla_x u(t,Tx)\vert^{q}\right) d\tau\nonumber\\
		&=\int_{0}^{s}r_{\alpha}(s-t) \left(\kappa_2\vert u(t,x)\vert^{\rho-1}u(t,x) -\kappa_1 \vert T \,\nabla_x u(t,T(x))\vert^{q}\right)d\tau=\theta(s,x)
	\end{align*}
	is symmetric whenever $u(t,\cdot)$ is also. Hence, 
	\begin{align}
		\mathcal{N}_{\alpha}(u)(t,x)&=\int_{0}^{t}\int_{\mathbb{R}^N}k_{\alpha,1}(t-s,x-y)\theta(s,y)dyds
	\end{align}
	is symmetric if $u(t,\cdot)$ is, for each $t>0$. From now on, employing an induction argument  in the following Picard's sequence
	\begin{align}
		u_{1}(t,x)& =G_{\alpha,2}(t)\psi +G_{\alpha,1}(t)\varphi\nonumber \\
		u_{k}(t,x)& =u_{1}(t,x)+\mathcal{N}_{\alpha }(u_{k-1})(t,x),\;k=2,3,\cdots
		\nonumber
	\end{align}%
	one can prove that $(u_k)$ is symmetric. It follows that $u(t,x)$ is symmetric, for
	all $t>0$. 	
\qed
%
%
	
%


\section*{References}
	
	\begin{thebibliography}{} 
		\bibitem{Marcelo} M. F. de Almeida\ and\ L. C. F. Ferreira, Self-similarity, symmetries and asymptotic behavior in Morrey spaces for a fractional wave equation, Differential Integral Equations {\bf 25} (2012), no.~9-10, 957--976.
		
		\bibitem{MJ} M. F. de Almeida\ and\ J. C. P. Precioso, Existence and symmetries of solutions in Besov--Morrey spaces for a semilinear heat-wave type equation, J. Math. Anal. Appl. {\bf 432} (2015), no.~1, 338--355.
		
		\bibitem{bsweissler} M. Ben-Artzi, P. Souplet\ and\ F. B. Weissler, The local theory for viscous Hamilton-Jacobi equations in Lebesgue spaces, J. Math. Pures Appl. (9) {\bf 81} (2002), no.~4, 343--378.
		
		\bibitem{Chipot}M. Chipot\ and\ F. B. Weissler, Some blowup results for a nonlinear parabolic equation with a gradient term, SIAM J. Math. Anal. {\bf 20} (1989), no.~4, 886--907.
		
	
		\bibitem{YF} Y. Fujita, Integrodifferential equation which interpolates the heat equation and the wave equation, Osaka J. Math. {\bf 27} (1990), no.~2, 309--321.
		
		\bibitem{YF1}Y. Fujita, Cauchy problems of fractional order and stable processes, Japan J. Appl. Math. {\bf 7} (1990), no.~3, 459--476.
		
		\bibitem{gild} B. H. Gilding, The Cauchy problem for $u\sb t=\Delta u+\vert \nabla u\vert \sp q$, large-time behaviour, J. Math. Pures Appl. (9) {\bf 84} (2005), no.~6, 753--785.
		
		\bibitem{HMiao} H. Hirata\ and\ C. Miao, Space-time estimates of linear flow and application to some nonlinear integro-differential equations corresponding to fractional-order time derivative, Adv. Differential Equations {\bf 7} (2002), no.~2, 217--236.
		
		
		\bibitem{Sawano}T. Izumi, Y. Sawano\ and\ H. Tanaka, Littlewood-Paley theory for Morrey spaces and their preduals, Rev. Mat. Complut. {\bf 28} (2015), no.~2, 411--447.
		
		\bibitem{Kato-Morrey} T. Kato, Strong solutions of the Navier-Stokes equation in Morrey spaces, Bol. Soc. Brasil. Mat. (N.S.) 22 (1992), no.~2, 127--155.
		
		\bibitem{Kilbas2} A.A. Kilbas, H.M. Srivastava, J.J. Trujillo, \textit{		Theory and Applications of Fractional Differential Equations}, North-Holland	Mathematics Studies, 204, Elsevier, Amsterdam, 2006.

		\bibitem{Yamazaki2}H. Kozono\ and\ M. Yamazaki, Semilinear heat equations and the Navier-Stokes equation with distributions in new function spaces as initial data, Comm. Partial Differential Equations {\bf 19} (1994), no.~5-6, 959--1014.
		
		\bibitem{Mazzucato} A. L. Mazzucato, Decomposition of Besov-Morrey spaces, in {\it Harmonic analysis at Mount Holyoke (South Hadley, MA, 2001)}, 279--294, Contemp. Math., 320, Amer. Math. Soc., Providence, RI. 
		
		\bibitem{Mazzucato2}A. L. Mazzucato, Besov-Morrey spaces: function space theory and applications to non-linear PDE, Trans. Amer. Math. Soc. {\bf 355} (2003), no.~4, 1297--1364.
		
		\bibitem{Miyakawa1} T. Miyakawa, On Morrey spaces of measures: basic
		properties and potential estimates, Hiroshima Math. J. 20 (1990), no.~1,
		213--222.
		
		\bibitem{Ribaud}F. Ribaud\ and\ A. Youssfi, Global solutions and self-similar solutions of semilinear wave equation, Math. Z. {\bf 239} (2002), no.~2, 231--262. 
		
		\bibitem{STW} S. Snoussi, S. Tayachi\ and\ F. B. Weissler, Asymptotically self-similar global solutions of a semilinear parabolic equation with a nonlinear gradient term, Proc. Roy. Soc. Edinburgh Sect. A {\bf 129} (1999), no.~6, 1291--1307.  
		
		\bibitem{Weissler-heat}S. Snoussi, S. Tayachi\ and\ F. B. Weissler, Asymptotically self-similar global solutions of a general semilinear heat equation, Math. Ann. {\bf 321} (2001), no.~1, 131--155. 
		
		\bibitem{Souplet1}P. Souplet, Finite time blow-up for a non-linear parabolic equation with a gradient term and applications, Math. Methods Appl. Sci. {\bf 19} (1996), no.~16, 1317--1333. 
		
		
		
		\bibitem{Stein1}E. M. Stein\ and\ G. Weiss, {\it Introduction to Fourier analysis on Euclidean spaces}, Princeton Univ. Press, Princeton, NJ, 1971.
		
		
	\end{thebibliography}
\end{document}